\documentclass[11pt,a4paper]{amsart}
\usepackage[bookmarks=false,colorlinks=false,pdfborder={0 0 0}]{hyperref}
\usepackage[utf8]{inputenc}
\usepackage{enumerate}
\usepackage{paralist}
\usepackage{mathrsfs}
\usepackage{color}
\usepackage[pdftex]{graphicx}
\usepackage{amssymb,amsmath,amsthm}
\usepackage{soul}
\usepackage{datetime}
\usepackage{titlesec}
\usepackage[outline]{contour}
\contourlength{0.2em}

\newcommand{\iref}[1]{\textit{(\ref{#1})}}

\DeclareGraphicsExtensions{pdf, jpg, jpeg, png}

\pdfpageattr{/Group << /S /Transparency /I true /CS /DeviceRGB>>}
\usepackage[a4paper,centering,includeheadfoot,margin=3cm]{geometry}

\newtheorem{theorem}{Theorem}
\newtheorem{lemma}[theorem]{Lemma}
\newtheorem{proposition}[theorem]{Proposition}
\newtheorem{corollary}[theorem]{Corollary}

\newtheorem{definition}[theorem]{Definition}

\theoremstyle{definition}
\newtheorem{remark}[theorem]{Remark}

\newtheorem*{remark*}{Remark}

\titleformat{\section}[hang]{\bfseries}{\thesection.}{.5em}{\centering}{}

\renewcommand{\paragraph}[1]{\medskip\noindent\textbf{#1}.}

\newcommand{\ommit}[1]{
}

\newcommand{\id}{\mathrm{id}}
\newcommand{\R}{\mathbb{R}}

\newcommand{\Z}{\mathbb{Z}}

\newcommand{\RP}{\mathbb{R}\textup{P}}

\newcommand{\coordsurf}[1]{\mathscr{B}^{#1}}
\newcommand{\cell}[1]{\mathscr{C}^{#1}}
\newcommand{\arcs}{\mathcal{A}^3}
\newcommand{\pr}{\mathcal{F}_{o,\pi}}
\newcommand{\linespace}{\mathcal{L}}

\newcommand{\vspan}[1]{\operatorname{span}(#1)}
\newcommand{\codim}{\operatorname{codim}}

\newcommand{\crr}{\operatorname{cr}}

\renewcommand{\phi}{\varphi}
\renewcommand{\epsilon}{\varepsilon}
\newcommand{\longto}{\longrightarrow}
\newcommand{\wtilde}{\widetilde}

\renewcommand{\implies}{\Longrightarrow}

\usepackage{caption}
\setcaptionwidth{.9\textwidth}

\title{Supercyclidic nets}

\author[A.~I.~Bobenko, E.~Huhnen-Venedey, and T.~R\"orig]{Alexander I. Bobenko \and Emanuel Huhnen-Venedey \and Thilo R\"orig}

\thanks{This research was supported by the DFG Collaborative Research Center TRR
109 “Discretization in Geometry and Dynamics” --
\texttt{http://www.discretization.de}.}

\address{Alexander I. Bobenko, Emanuel Huhnen-Venedey, Thilo
R\"orig\vspace{-.5em}}
\address{Institute of Mathematics, Secr. MA 8-4, TU Berlin, 10623
Berlin, Germany\vspace{-.5em}}
\address{Email: {\normalfont \{bobenko,huhnen,roerig\}@math.tu-berlin.de}}



\begin{document}

\begin{abstract}
	Supercyclides are surfaces with a characteristic conjugate parametrization
	consisting of two families of conics. Patches of supercyclides can be adapted
	to a Q-net (a discrete quadrilateral net with planar faces) such that
	neighboring surface patches share tangent planes along common boundary curves.
	We call the resulting patchworks ``supercyclidic nets'' and show that every
	Q-net in $\RP^3$ can be extended to a supercyclidic net. The construction is
	governed by a multidimensionally consistent 3D system. One essential aspect of
	the theory is the extension of a given Q-net in $\RP^N$ to a system of
	circumscribed discrete torsal line systems. We present a description of the
	latter in terms of projective reflections that generalizes the systems of
	orthogonal reflections which govern the extension of circular nets to cyclidic
	nets by means of Dupin cyclide patches.
\end{abstract}

\maketitle

\section{Introduction}

Discrete differential geometry aims at the development of discrete equivalents
of notions and methods of classical differential geometry.  One prominent
example is the discretization of parametrized surfaces and the related theory,
where it is natural to discretize parametrized surfaces by quadrilateral nets
(also called quadrilateral meshes). In contrast to other discretizations of
surfaces as, e.g., triangulated surfaces, a quadrilateral mesh reflects the
combinatorial structure of parameter lines.  While unspecified quadrilateral
nets discretize arbitrary parametrizations, the discretization of distinguished
types of parametrizations yields quadrilateral nets with special properties. 

The present work is on the piecewise smooth discretization of classical
conjugate nets by \emph{supercyclidic nets}. They arise as an extension of the
well established, integrable discretization of conjugate nets by quadrilateral
nets with planar faces, the latter often called \emph{Q-nets} in discrete
differential geometry.  Two-dimensional Q-nets as discrete versions of conjugate
surface parametrizations were proposed by Sauer in the 1930s
\cite{Sauer:1937:ProjLinienGeometrie}. Multidimensional Q-nets are a subject of
modern research \cite{DoliwaSantini:1997:QnetsAreIntegrable,BobenkoSuris:2008:DDGBook}.
The surface patches that we use for the extension are pieces of
\emph{supercyclides}, a class of surfaces in projective 3-space that we discuss
in detail in Section~\ref{sec:supercyclides}. Supercyclides possess conjugate
parametrizations with all parameter lines being conics and such that there
exists a quadratic tangency cone along each such conic.  As a consequence,
isoparametrically bounded surface patches coming from those characteristic
parametrizations (referred to as \emph{SC-patches}) always have coplanar
vertices, which makes them suitable for the extension of Q-nets to piecewise
smooth objects. The 2-dimensional case of such extensions has been proposed
previously in the context of Computer Aided Geometric Design (CAGD)
\cite{Pratt:1996:DupinCyclidesAndSupercyclides,
  Pratt:2002:QuarticSupercyclidesDesign,
  Degen:1994:GeneralizedCyclidesCAGD},
but to the best of our knowledge has not been worked out so far. 

Based on established notions of discrete differential geometry,
we describe in this article the multidimensionally consistent, piecewise smooth
extension of $m$-dimensional Q-nets in $\RP^3$ by adapted surface patches, such
that edge-adjacent patches in one and the same coordinate plane have coinciding
tangent planes along a common boundary curve (see Fig.~\ref{fig:sc_net_intro} for a
2-dimensional example). Relevant aspects of the existing theory are presented in
Sections~\ref{sec:qnets_line_systems}~to~\ref{sec:supercyclides} and
succeeded by new results which are organized as follows:
\begin{itemize}
\item We describe the extension of 2D Q-nets to supercyclidic
nets in Section~\ref{sec:2d_supercyclidic_nets}, emphasizing the underlying 2D system
that governs the extension of Q-nets to fundamental line systems.
\item The 3D system that governs multidimensional supercyclidic nets
	is uncovered and analyzed in Section~\ref{sec:3d_supercyclidic_nets}. We also present
	piecewise smooth conjugate coordinate systems that are locally induced by 3D
	supercyclidic nets and give rise to arbitrary Q-refinements of their support
	structures.
\item We define $m$D supercyclidic nets and develop a transformation theory
thereof that appears as a combination of the existing smooth and discrete
theories in Section~\ref{sec:md_supercyclidic_nets}.
\item Finally, we introduce frames of supercyclidic nets and describe a related
integrable system on projective reflections in Section~\ref{sec:frames}.
\end{itemize}

\begin{figure}[htb]
\begin{center}
\parbox{.42\textwidth}{\includegraphics[width=.42\textwidth]{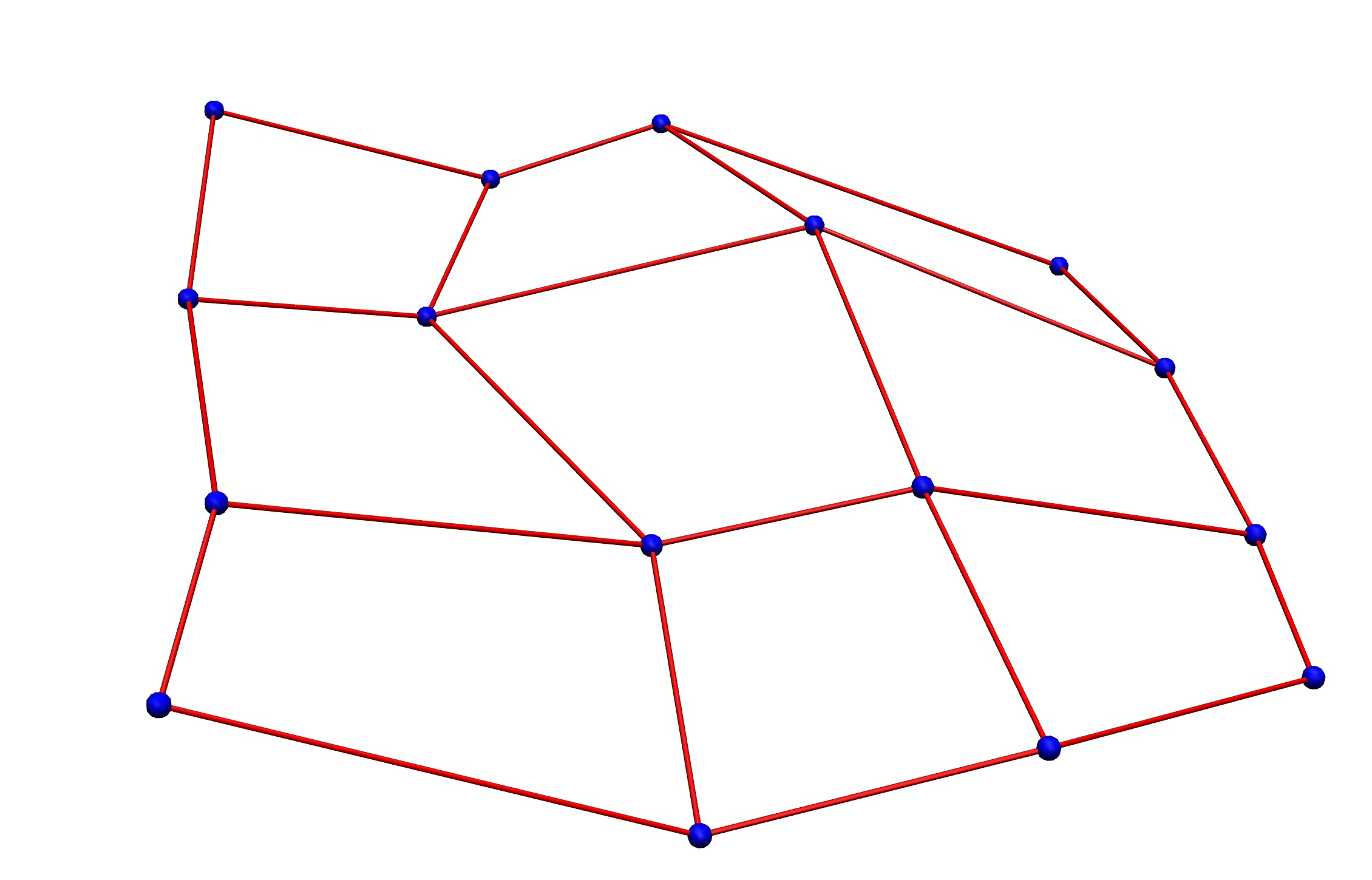}}
\hspace{.07\textwidth}
\parbox{.42\textwidth}{\includegraphics[width=.42\textwidth]{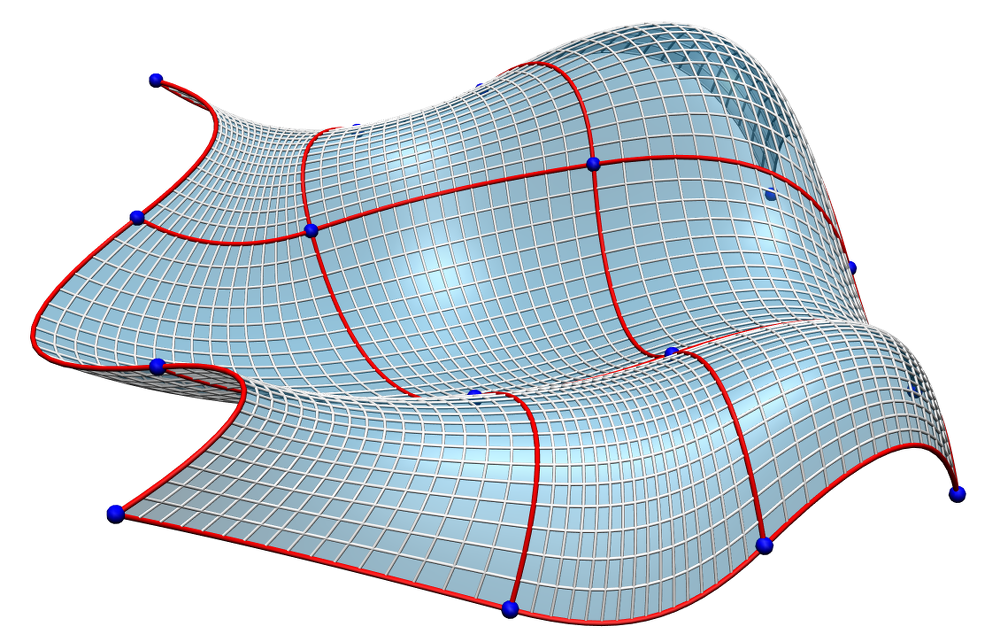}}
\end{center}
\caption{A Q-net and its extension to a supercyclidic net. Each elementary
quadrilateral is replaced by an adapted SC-patch, that is, a supercyclidic patch
whose vertices match the supporting quadrilateral.}
\label{fig:sc_net_intro}
\end{figure}

The present work is a continuation of our previous works on the discretizations
of smooth orthogonal and asymptotic nets by \emph{cyclidic} and \emph{hyperbolic
  nets}, respectively \cite{BobenkoHuhnen-Venedey:2011:cyclidicNets,
  Huhnen-VenedeyRoerig:2013:hyperbolicNets, Huhnen-VenedeySchief:2014:wTrafos}.
We will now review the corresponding theory for cyclidic nets and illustrate the
core concepts that may serve as a structural guideline for this article.

\paragraph{Cyclidic nets} 
Three-dimensional orthogonal nets in $\R^3$ are triply orthogonal coordinate
systems, while 2-di\-men\-sion\-al orthogonal nets are curvature line
parametrized surfaces. Based on the well-known discretization of orthogonal nets
by circular nets
\cite{Bobenko:1999:DiscreteConformalMaps,CieslinskiDoliwaSantini:1997:CircularNets,KonopelchenkoSchief:1998:3DIntegrableLattices,BobenkoMatthesSuris:2003:OrthogonalSystems},
cyclidic nets are defined as circular nets with an additional structure, that
is, as circular nets with orthonormal frames at vertices that determine a unique
adapted \emph{DC-patch} for each elementary quadrilateral with the prescribed
circular points as vertices. The term ``DC-patches'' refers to surface patches
of \emph{Dupin cyclides}\footnote{Dupin cyclides are surfaces in 3-space that
are characterized by the property that all curvature lines are circles.  They
are special instances of supercyclides.}  that are bounded by
curvature lines, i.e., circular arcs.  For a 2-dimensional cyclidic net
its frames are such that the adapted DC-patches constitute a piecewise
smooth $C^1$-surface, cf.  Fig.~\ref{fig:cyclidic_from_circular_intro}.

\begin{figure}[htb]
\begin{center}
\parbox{.26\textwidth}{\includegraphics[width=.26\textwidth]{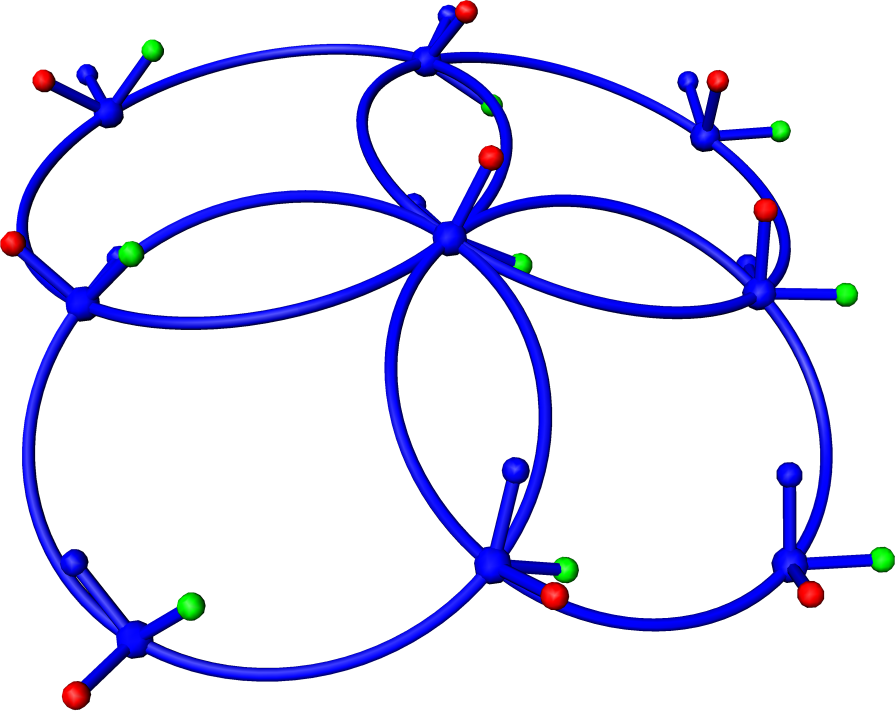}}
\hspace{.07\textwidth}
$\leftrightarrow$
\hspace{.08\textwidth}
\parbox{.26\textwidth}{\includegraphics[width=.26\textwidth]{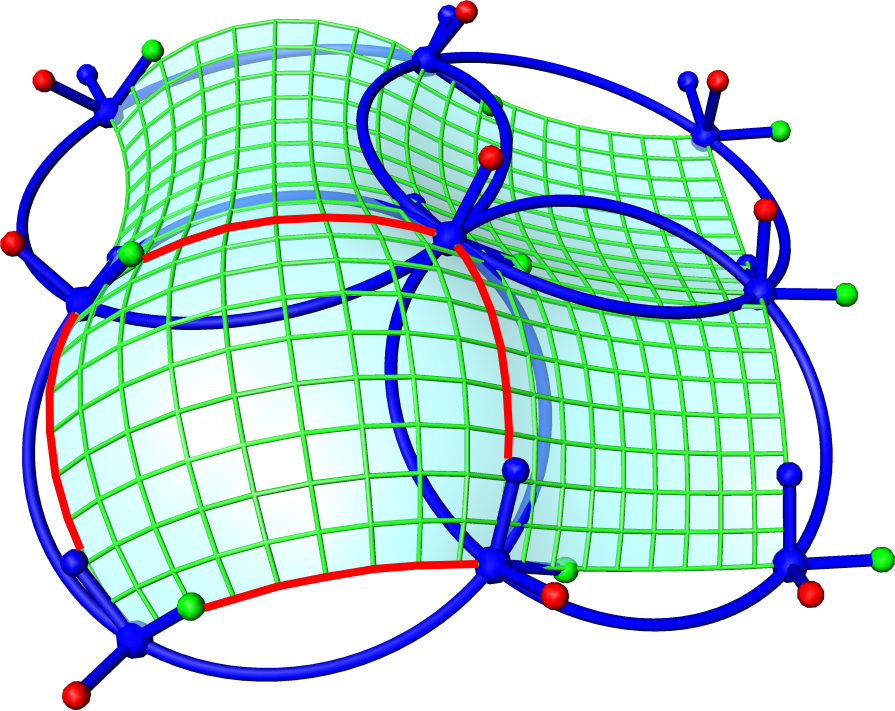}}
\end{center}
\caption{A 2-dimensional cyclidic net may be understood as a piecewise smooth $C^1$-surface
composed of DC-patches.}
\label{fig:cyclidic_from_circular_intro}
\end{figure}

For any DC-patch the tangent planes at the vertices are tangent to a common cone
of revolution.  Therefore, the tangent planes to a 2D cyclidic net at the
vertices constitute a \emph{conical
  net}~\cite{LiuPottmannWallnerYangWang:2006:ConicalNets} and combination with
the (concircular) vertices yields a \emph{principal contact element net}, so
that cyclidic nets comprise the recognized discretizations of curvature line
parametrizations by conical and principal contact element nets
\cite{BobenkoSuris:2008:DDGBook,PottmannWallner:2007:FocalGeometry}.

Going beyond the discretization of curvature line parametrized surfaces, higher
dimensional cyclidic nets provide a discretization of orthogonal coordinate
systems that is motivated by the classical Dupin theorem. The latter states that the
coordinate surfaces of a triply orthogonal coordinate system are curvature line
parametrized surfaces. Accordingly, a 3-dimensional cyclidic net is a
3D circular net with frames at vertices that describe 2D cyclidic nets in each
coordinate plane and such that for each edge of $\Z^3$ one has one unique
circular arc that is a shared boundary curve of all adjacent DC-patches, cf.
Fig.~\ref{fig:3d_cyclidic}.

\begin{figure}[htb]
\begin{center}
\includegraphics[width=.3\textwidth]{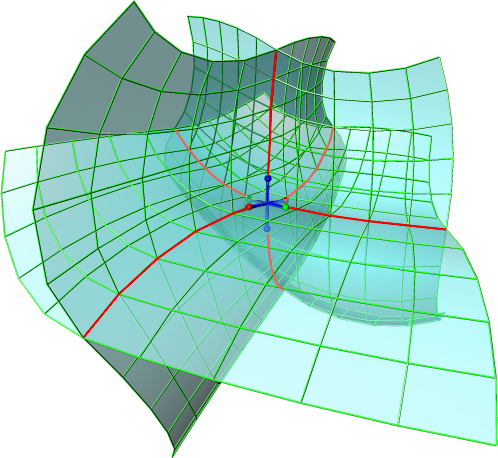}
\end{center}
\caption{Three coordinate surfaces of a 3-dimensional cyclidic net.}
\label{fig:3d_cyclidic}
\end{figure}

Cyclidic nets are piecewise smooth extensions of discrete integrable support
structures by surface patches of a particular class that is in some sense as
easy as possible but at the same time as flexible as necessary: We
use DC-patches for the discretizaton of curvature line parametrized surfaces and
those patches are actually characterized by the geometric property that all
curvature lines are circular arcs, that is, the simplest non-straight curves.
At the same time, DC-patches are flexible enough in order to be fit together to
adapted $C^1$-surfaces. Moreover, the class of DC-patches is preserved by
M\"obius transformations, which is in line with the transformation group
principle for the discretization of integrable
geometries~\cite{BobenkoSuris:2008:DDGBook}.  But there is also a deeper reason
to choose Dupin cyclides for the extension of circular nets, that is, they
actually \emph{embody the geometric characterizations} of the latter in the
following sense: It is not difficult to see that Dupin cyclides are
characterized by the fact that arbitrary selections of curvature lines
constitute a principal contact element net (a circular net with normals at
vertices that are related by reflections in symmetry planes).  As a consequence,
cyclidic nets induce arbitrary refinements of their support structures within
the class of circular nets, by simply adding parameter lines of the adapted
patches to the support structures. After all, Dupin cyclides provide a perfect
link between the theories of smooth and discrete orthogonal nets. This is also
reflected by the theory of transformations of cyclidic nets, which arises by
combination of the corresponding smooth and discrete theories.
\emph{It turns out that supercyclidic patches play an analogous role for
the piecewise smooth extension of Q-nets.}

Analogous to the theory of cyclidic nets is the theory of hyperbolic nets as
piecewise smooth discretizations of surfaces in 3-space that are parametrized
along asymptotic lines. On a purely discrete level, such nets are properly
discretized by quadrilateral nets with planar vertex stars
\cite{Sauer:1937:ProjLinienGeometrie, BobenkoSuris:2008:DDGBook}.  Based on that
discretization, hyperbolic nets arise as an extension of A-nets by means of
hyperbolic surface patches, which is analogous to the extension of 2D circular
nets to 2D cyclidic nets via DC-patches. Many results concerning cyclidic nets
on the one hand and hyperbolic nets on the other hand correspond to each other
under Sophus Lie's line-sphere-correspondence.

\paragraph{Supercyclides in CAGD and applications of supercyclidic nets}
Supercyclides have been introduced as \emph{double Blutel surfaces} by Degen in
the 1980s \cite{Degen:1982:ConjugateConics,Degen:1986:ZweifachBlutel}.  Soon
they found attention in Computer Aided Geometric Design for several reasons. For
example their pleasant geometric properties, nicely reviewed in
\cite{Degen:1994:GeneralizedCyclidesCAGD}, allow to use parts of supercyclides
for blends between quadrics
\cite{AllenDutta:1997:SupercyclidesBlending,FoufouGarnier:2005:ImplicitEuqationsOfSupercyclides}.
Supercyclidic patchworks with differentiable joins have been proposed in the
past
\cite{Pratt:1996:DupinCyclidesAndSupercyclides,Degen:1994:GeneralizedCyclidesCAGD}
as a continuation of the discussion of composite $C^1$-surfaces built from
DC-patches within the CAGD community (see, e.g.,
\cite{Martin:1983:PrincipalPatches,McLean:1985:CyclideSurfaces,MartinDePontSharrock:1986:CyclideSurfaces,DuttaMartinPratt:1993:CyclideSurfaceModeling,SrinivasKumarDutta:1996:SurfaceDesignUsingCyclidePatches}
for the latter).

Another field of potential application for supercyclidic nets is architectural
geometry (see the comprehensive book
\cite{PottmannAsperlHoferKilian:2007:ArchitecturalGeometry} for an introduction
to that synthesis of architecture and geometry), because they are aesthetically
appealing freeform surfaces that can be approximated at arbitrary precision by
flat panels. On the other hand, supercyclidic nets extend their supporting Q-net
and may be seen as a kind of 3D texture for that support structure. Purely Dupin
cyclidic nets in 3-space, which are special instances of supercyclidic nets,
have already been discussed in the context of \emph{circular arc structures} in
\cite{BoEtAl:2011:CAS}, while
\cite{KaeferboeckPottmann:2012:DiscreteAffineMinimal,ShiWangPottmann:2013:RationalBilinearPatches}
provides an analysis of hyperbolic nets that aims at the application thereof in
architecture.

\section{Q-nets and discrete torsal line systems}
\label{sec:qnets_line_systems}

Q-nets are discrete versions of classical conjugate nets and closely related to
discrete torsal line systems. Before starting our discussion of those discrete
objects, we recall a characterization of classical conjugate nets and their
transformations that can be found in \cite{Eisenhart:1923:Transformations} or
\cite[Chap.~1]{BobenkoSuris:2008:DDGBook}.

\begin{definition}[Conjugate net]
\label{def:conjugate_net}
A map $x : \R^m \to \R^N,\, N \ge 3$, is called an \emph{$m$-dimensional
conjugate net in $\R^N$} if at every point of the domain and for all pairs $1
\le i \ne j \le m$ one has $\partial_{ij}x \in \vspan{\partial_i x, \partial_j
x} \iff \partial_{ij}x = c_{ji} \partial_i x + c_{ij} \partial_j
x$.\footnote{Loosely speaking, infinitesimal quadrilaterals formed by parameter
lines of a conjugate net are planar.}
\end{definition}

\begin{definition}[F-transformation of conjugate nets]
\label{def:f-trafo_smooth}
Two $m$-dimensional conjugate nets $x,x^+$ are said to be related by a
\emph{fundamental transformation} (\emph{F-transformation}) if at every point
of the domain and for each $1 \le i \le m$ the three vectors
$\partial_i x, \partial_i x^+,$ and $\delta x = x^+ - x$ are coplanar. The net
$x^+$ is called an \emph{F-transform} of the net $x$.
\end{definition}

The above F-transformations of conjugate nets exhibit the following,
Bianchi-type permutability properties. The existence of associated
transformations with permutability properties of that kind is classically
regarded as a key feature of integrable systems.

\begin{theorem}[Permutability properties of F-transformations of
conjugate nets]
\label{thm:permutability_smooth}
\hfill
\begin{enumerate}[(i)]
\item 
	Let  $x$ be an $m$-di\-men\-sional conjugate net, and let $x^{(1)}$ and
	$x^{(2)}$ be two of its F-transforms.  Then there exists a 2-parameter family
	of conjugate nets $x^{(12)}$ that are F-transforms of both $x^{(1)}$ and
	$x^{(2)}$.  The corresponding points of the four conjugate nets $x$,
	$x^{(1)}$, $x^{(2)}$ and $x^{(12)}$ are coplanar.
\item 
	Let  $x$ be an $m$-dimensional conjugate net. Let $x^{(1)}$, $x^{(2)}$ and
	$x^{(3)}$ be three of its F-transforms, and let three further
	conjugate nets $x^{(12)}$, $x^{(23)}$ and $x^{(13)}$ be given such that
	$x^{(ij)}$ is a simultaneous F-transform of $x^{(i)}$ and $x^{(j)}$.
	Then generically there exists a unique conjugate net $x^{(123)}$ that is an
	F-transform of $x^{(12)}$, $x^{(23)}$ and $x^{(13)}$.  The net
	$x^{(123)}$ is uniquely defined by the condition that for every permutation
	$(ijk)$ of $(123)$ the corresponding points of $x^{(i)}$, $x^{(ij)}$,
	$x^{(ik)}$ and $x^{(123)}$ are coplanar.
\end{enumerate}
\end{theorem}

Although we gave an affine description, conjugate nets and their transformations
are objects of projective differential geometry.  Accordingly, we consider the
theory of discrete conjugate nets in $\RP^N$ and not in $\R^N$. Before coming to
that, we explain some notation that will be used throughout this article.

\paragraph{Notation for discrete maps}
Discrete maps are fundamental in discrete differential geometry. We are
mostly concerned with discrete maps defined on cells of dimension
$0,1$, or $2$ of $\Z^m$, that is, maps defined on vertices, edges, or elementary
quadrilaterals.

Let $e_1,\ldots,e_m$ be the canonical basis of the m-dimensional
lattice~$\Z^m$. For $k$ pairwise distinct indices $i_1,\ldots,i_k \in
\left\{ 1,\dots,m \right\}$ we denote by
\begin{equation*}
\coordsurf{i_1 \dots i_k} = \operatorname{span}_\Z(e_{i_1},\dots,e_{i_k})
\end{equation*}
the $k$-dimensional coordinate plane of $\Z^m$ and by
\begin{equation*}
\cell{i_1 \dots i_k}(z) = \left\{ z + \varepsilon_{i_1} e_{i_1} + \dots +
\varepsilon_{i_k} e_{i_k} \mid \varepsilon_{i} = 0,1 \right\}
\end{equation*}
the $k$-cell at~$z$ spanned by $e_{i_1},\ldots,e_{i_k}$, respectively.

We use upper indices $i_1,\dots,i_k$
to describe maps on $k$-cells as maps on $\Z^m$ by identifying the
$k$-cell $\cell{i_1 \dots i_k}(z)$ with its basepoint $z$
\begin{equation*}
f^{i_1 \dots i_k}(z) = f(\cell{i_1 \dots i_k}(z)).
\end{equation*}
For a map $f$ defined on $\Z^m$ we use lower indices to indicate shifts in
coordinate directions 
\begin{equation*}
f_i(z) = f(z+e_i), \quad
f_{ij}(z) = f(z + e_i + e_j), \quad \dots
\end{equation*}
Often we omit the argument of $f$, writing
\begin{equation*}
f = f (z), \quad f_i = f (z + ei), \quad f_{-i} = f(z-e_i), \quad \dots
\end{equation*}
and analogous for maps defined on cells of higher dimension.

\paragraph{Q-nets and discrete torsal line systems}
Smooth conjugate nets are discretized within discrete differential geometry by
quadrilateral meshes with planar faces. The planarity condition is a straight
forward discretization of the smooth characteristic property $\partial_{ij} f
\in \vspan{\partial_i f, \partial_j f}$. In the 2-dimensional case, this
discretization has been proposed by Sauer~\cite{Sauer:1937:ProjLinienGeometrie}
and was later generalized to the multidimensional case
\cite{DoliwaSantini:1997:QnetsAreIntegrable,BobenkoSuris:2008:DDGBook}.

\begin{definition}[Q-net]
\label{def:qnet}
A map $x:\Z^m \to \RP^N$, $N\ge3$, is called an \emph{m-dimensional Q-net} or
\emph{discrete conjugate net} in $\RP^N$ if
for all pairs $1 \le i < j \le m$ the elementary quadrilaterals
$(x,x_i,x_{ij},x_j)$ are planar.
\end{definition}

Closely related to Q-nets in $\RP^N$ are configurations of lines in $\RP^N$
with, e.g., $\Z^m$ combinatorics, such that neighbouring lines intersect.  We
denote the manifold of lines in $\RP^N$ by
\begin{equation*}
\linespace^N := \left\{ \text{Lines in } \RP^N \right\} \cong
\operatorname{Gr}(2,\R^{N+1}),
\end{equation*}
where $\operatorname{Gr}(2,\R^{N+1})$ is the Grassmanian of 2-dimensional linear
subspaces of $\R^{N+1}$.

\begin{definition}[Discrete torsal line system]
\label{def:line_congruence}
A map $l : \Z^m \to \linespace^N$, $N \ge 3$, is called an \emph{m-dimensional
discrete torsal line system} in $\RP^N$ if at each $z \in \Z^m$ and for all $1
\le i \le m$ the neighbouring lines $l$ and $l_i$ intersect.
We say that $l$ is \emph{generic} if
\begin{enumerate}[(i)]
\item For each elementary quadrilateral of $\Z^m$ the lines associated with
  opposite vertices are skew (and therefore span a unique 3-space that contains
	all four lines of the quadrilateral).
\item
If $m \ge 3$ and $N \ge 4$, the space spanned by any quadruple
$(l,l_i,l_j,l_k)$ of lines, $1 \le i < j < k \le m$, is 4-dimensional.
\end{enumerate}
A 2-dimensional line system is called a \emph{line congruence} and a
3-dimensional line system is called a \emph{line complex}.
\end{definition}

\enlargethispage{\baselineskip}

Recently, line systems on triangle meshes with non-intersecting neighboring
lines were studied
\cite{WangJiangBompasWallnerPottmann-2013-DiscreteLineConSGP}. To distinguish
the two different types of line systems we call the systems with intersecting
neighboring lines \emph{torsal}. Discrete torsal line systems and Q-nets are
closely related \cite{DoliwaSantiniManas:2000:TransformationsOfQnets}.

\begin{definition}[Focal net]
\label{def:focal_net}
For a discrete torsal line system $l : \Z^m \to \linespace^N$ and a direction
$i \in \left\{ 1,\ldots,m \right\}$, the \emph{i-th focal net} $f^i : \Z^m \to
\RP^N$ is defined by
\begin{equation*}
f^i(z) := l(z) \cap l(z + e_i).
\end{equation*}
The planes spanned by adjacent lines of the system are called \emph{focal
planes}. Focal points and focal planes of a discrete torsal line system are
naturally associated with edges of $\Z^m$, cf.~Fig.~\ref{fig:elem_quad_line_congruence}.
\end{definition}

\begin{figure}[ht]
 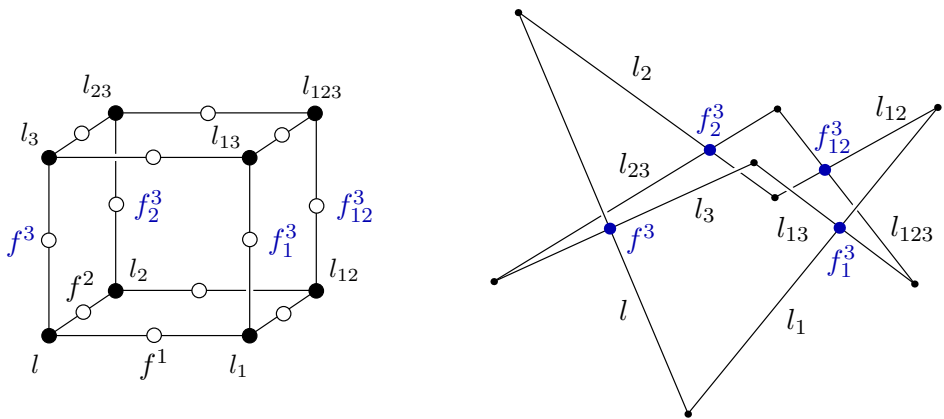 
\caption{An elementary cube of a discrete torsal line system:
combinatorial and geometric.}
\label{fig:elem_quad_line_congruence}
\end{figure}

\begin{definition}[Edge systems of Q-nets and Laplace transforms]
\label{def:laplace}
Given a Q-net $x : \Z^m \to \RP^N$ and a direction $i \in \left\{ 1,\ldots,m
\right\}$, we say that the extended edges of direction~$i$ constitute the
\emph{$i$-th edge system} $e^i : \Z^m \to \linespace^N$. By definition, the
$i$-th edge system is a discrete torsal line system, whose different focal nets
are called \emph{Laplace transforms} of $x$.  Accordingly, the vertices of the
different Laplace transforms are called \emph{Laplace points} of certain
directions with respect to the elementary quadrilaterals of $x$.
\end{definition}

Edge systems are called \emph{tangent systems}
in~\cite{DoliwaSantiniManas:2000:TransformationsOfQnets}. We prefer to call them
edge systems to distinguish them from the tangents systems of supercyclidic nets
(formally defined only in Section~\ref{sec:frames}), which consist of tangents to
surface patches at vertices of a supporting Q-net.

\smallskip

From simple dimension arguments one obtains the following

\begin{theorem}[Focal nets are Q-nets]
\label{thm:focal_nets_are_qnets}
For a generic discrete torsal line system $l : \Z^m \to \linespace^N$, $N \ge
4$, each focal net is a Q-net in $\RP^N$.
Moreover, focal quadrilaterals of the type $(f^i,f^i_i,f^i_{ij},f^i_j)$
are planar for arbitrary $N$.\footnote{For $m \ge 3$ there are focal
quadrilaterals of the type $(f^i,f^i_j,f^i_{jk},f^i_k)$, $i \ne j \ne k
\ne i$, which are generically not planar in the case $N = 3$.}
\end{theorem}

\paragraph{Underlying 3D systems that govern Q-nets and discrete torsal line
systems}\\
Q-nets in $\RP^N, N \ge 3$, and also discrete torsal line systems in
$\RP^N, N \ge 4$, are each governed by an associated discrete 3D system:
generic data at seven vertices of a cube determine the remaining 8th
value uniquely, cf. Fig.~\ref{fig:3d_system}. In both cases, this is evident if one
considers intersections of subspaces and counts the generic dimensions with
respect to the ambient space. For example, consider the seven vertices
$x,x_i,x_{ij}$, $i,j = 1,2,3$, $i \ne j$, of an elementary cube of a Q-net in
$\RP^N$. Those points necessarily lie in a 3-dimensional subspace and determine
the value $x_{123}$ uniquely as intersection point of the three planes
$\pi^{12}_3 = x_3 \vee
x_{13} \vee x_{23}$,\ $\pi^{23}_1 = x_1 \vee x_{12} \vee x_{13}$, and
$\pi^{13}_2 = x_2 \vee x_{12} \vee x_{23}$.

\begin{figure}[htb]
 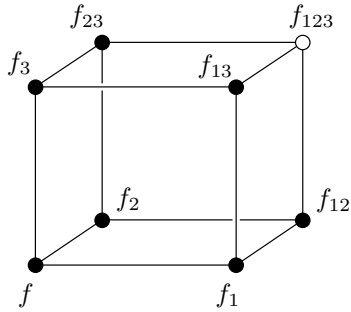 
\caption{The seven values $f,f_i,f_{ij}$ (vertices of a Q-net or lines of a
discrete torsal line system) determine the remaining value $f_{123}$ uniquely.}
\label{fig:3d_system}
\end{figure}

3D systems of the above type allow to propagate Cauchy data
\begin{equation*}
f|_{\coordsurf{ij}}, \quad 1 \le i < j \le 3
\end{equation*}
on the coordinate planes $\coordsurf{12},\coordsurf{23},\coordsurf{13}$ to the
whole of $\Z^3$ uniquely. In higher dimensions, the fact that the a~priori
overdetermined propagation of Cauchy data
\begin{equation*}
f|_{\coordsurf{ij}}, \quad 1 \le i < j \le m
\end{equation*}
to the whole of $\Z^m$ is well-defined is referred to as \emph{multidimensional
consistency} of the underlying system and understood as its \emph{discrete
integrability}. Simple combinatorial considerations show that for
discrete $m$D systems that allow to determine the value at one vertex of an $m$D
cube from the values at the remaining vertices, the $(m+1)$D consistency
implies $(m+k)$D consistency for arbitrary $k \ge 1$. Indeed, this is the case
for the systems governing Q-nets and discrete torsal line systems, cf.
\cite{BobenkoSuris:2008:DDGBook}. For future reference, we capture the above in

\begin{theorem}
  \label{thm:4d_consistency}
	Discrete torsal line systems in $\RP^N$, $N \ge 4$, as well as Q-nets in
	$\RP^M$, $M \ge 3$, are each governed by $m$D consistent 3D systems.
\end{theorem}

Multidimensional consistency assures the existence of associated transformations
that exhibit Bianchi-type permutability properties in analogy to the
corresponding classical integrable systems. In fact, it is a deep result of
discrete differential geometry that on the discrete level discrete integrable
nets and their transformations become indistinguishable. A prominent example for
this scheme are F-transformations of Q-nets. We simply discretize the defining
property captured in Definition~\ref{def:f-trafo_smooth} by replacing the
partial derivatives $\partial_i x, \partial_i x^+$ by difference vectors.

\begin{definition}[F-transformation of Q-nets]
\label{def:qnets_ftransform}
Two $m$-dimensional Q-nets
\begin{equation*}
x,x^+ : \Z^m \to \RP^N
\end{equation*}
are called \emph{F-transforms} (\emph{fundamental transforms}) of one another
if at each $z \in \Z^m$ and for all $1 \le i \le m$ the quadrilaterals $(x, x_i,
x_i^+, x^+)$ are planar.
\end{definition}

Obviously, the condition in Definition~\ref{def:qnets_ftransform} may be re-phrased as follows:
The Q-nets $x$ and $x^+$ are F-transforms of one another if the net
$X : \Z^m \times \left\{ 0,1 \right\} \to \RP^N$
defined by 
\begin{equation*}
X(z,0) = x(z) \text{ and } X(z,1) = x^+(z)
\end{equation*}
is a two-layer $(m+1)$-dimensional Q-net. Therefore, the existence of
F-transforms of Q-nets with permutability properties analogous to the classical
situation is a simple consequence of the multidimensional consistency of Q-nets.
One obtains

\begin{theorem}[Permutability properties of F-transformations of Q-nets]
\label{thm:permutability_qnets}
\hfill
\begin{enumerate}[(i)]
\item 
Let  $x$ be an $m$-dimensional Q-net, and let $x^{(1)}$ and $x^{(2)}$ be two of
its discrete F-transforms. Then there exists a 2-parameter family of Q-nets
$x^{(12)}$ that are discrete F-transforms of both $x^{(1)}$ and $x^{(2)}$. The
corresponding points of the four Q-nets $x$, $x^{(1)}$, $x^{(2)}$ and
$x^{(12)}$ are coplanar. The net $x^{(12)}$ is uniquely determined by one of
its points.
\item
Let  $x$ be an $m$-dimensional Q-net. Let $x^{(1)}$, $x^{(2)}$ and $x^{(3)}$ be
three of its discrete F-transforms, and let three further Q-nets $x^{(12)}$,
$x^{(23)}$ and $x^{(13)}$  be given such that $x^{(ij)}$ is a simultaneous
discrete F-transform of $x^{(i)}$ and $x^{(j)}$. Then generically there exists
a unique Q-net $x^{(123)}$ that is a discrete F-transform of $x^{(12)}$,
$x^{(23)}$ and $x^{(13)}$. The net $x^{(123)}$ is uniquely determined by the
condition that for every permutation $(ijk)$ of $(123)$ the corresponding
points of $x^{(i)}$, $x^{(ij)}$, $x^{(ik)}$ and $x^{(123)}$ are coplanar.
\end{enumerate}
\end{theorem}

Theorem~\ref{thm:permutability_qnets} yields the classical results captured by
Theorem~\ref{thm:permutability_smooth} after performing a refinement limit in the ``net
directions'' while keeping the ``transformation directions'' discrete.

\section{Fundamental line systems}
\label{sec:fundamental_line_systems}

While generic discrete torsal line systems in $\RP^N, N\ge 4$, are governed by a
3D system, this is not the case in $\RP^3$. The reason is that for three generic
lines $l_{12}, l_{23}, l_{13}$ in 4-space there is a unique line that intersects
all of them, while in $\RP^3$ there is a whole 1-parameter family of such lines.
However, one obtains a 3D system for line systems in $\RP^3$ by
demanding that the focal quadrilaterals be planar, which turns out to be an
admissible reduction of the considered configurations
(Corollary~\ref{cor:2d_quad_line_systems}).

\begin{definition}[Fundamental line system]
\label{def:fundamental_line_system}
A discrete torsal line system is called \emph{fundamental} if its focal nets
are Q-nets and if it is generic in the sense that lines associated with
opposite vertices of elementary quadrilaterals are skew. In particular, any
generic discrete torsal line system in $\RP^N$, $N \ge 4$, is fundamental
according to Theorem~\ref{thm:focal_nets_are_qnets}.
\end{definition}

\begin{remark*}
The terminology was introduced in~\cite{BobenkoSchief:2014:FundamentalSystems},
where fundamental line systems are discussed in the context of integrable
systems.  It is shown that elementary cubes of such line systems in $\RP^3$ are
characterized by the existence of a unique involution $\tau$ on $\linespace^3$,
such that lines associated with opposite vertices of the cube are interchanged.
The map $\tau$ is determined by any six lines that remain if one removes a pair
of opposite lines. For example, for fixed lines
$l_1,l_2,l_3,l_{12},l_{23},l_{13}$ one may express $l_{123}$ as dependent on $l$
according to $l_{123} = \tau(l)$.  Moreover, it turns out that $\tau$ comes
from a \emph{correlation} of $\RP^3$ (a projective map from $\RP^3$ to its dual,
that is, points and planes are interchanged but lines are mapped to lines).
\end{remark*}

In the following we prove essential properties of fundamental line systems that
will play a key role for our later considerations of supercyclidic nets.  Many
of the presented statements appear
in~\cite[Sect.~2]{DoliwaSantiniManas:2000:TransformationsOfQnets}

\begin{lemma}
\label{lem:fundamental_cubes}
Elementary cubes of a fundamental line system are characterized by any of the
following properties:
\begin{enumerate}[(i)]
\item
\label{item:lem_fun_cubes_planar_quad_one}
One focal quadrilateral is planar.
\item
\label{item:lem_fun_cubes_planar_quad_any}
All focal quadrilaterals are planar.
\item
\label{item:lem_fun_cubes_concurrent_planes_one}
The four focal planes of one direction are concurrent, that is, intersect
in one point.
\item
\label{item:lem_fun_cubes_concurrent_planes_any}
For each direction the four associated focal planes are concurrent.
\end{enumerate}
\end{lemma}

\begin{proof}
	If the $i$-th focal quadrilateral $Q^i = (f^i, f^i_j, f^i_k, f^i_{jk})$ is
	planar \iref{item:lem_fun_cubes_planar_quad_one} then the $j$-th focal
	planes all pass through the Laplace point of $Q^i$ in direction $j$ and
	analogous for the $k$-th focal planes
	\iref{item:lem_fun_cubes_concurrent_planes_one}. Hence also the $j$-th
	and $k$-th focal quadrilaterals are planar, i.e., all focal quadrilaterals are
	planar \iref{item:lem_fun_cubes_planar_quad_any}. This in turn implies
	that also the $i$-th focal planes are concurrent and thus each family of
	focal planes is concurrent
	\iref{item:lem_fun_cubes_concurrent_planes_any}.  Finally, this
	implies \iref{item:lem_fun_cubes_planar_quad_one}.
\end{proof}

\begin{proposition}
\label{prop:fundamental_line_systems}
If one focal net of a discrete torsal line system is a Q-net,
the line system is fundamental.
\end{proposition}

\begin{proof}
  If the line system is at least 3-dimensional, there are two types of focal
  quadrilaterals of the $i$-th focal net: $Q^{ij} = (f^i,f^i_i,f^i_{ij},f^i_j)$
  and $Q^{jk} = (f^i,f^i_j,f^i_{jk},f^i_k)$, $i \ne j \ne k \ne i$. The vertices
  of $Q^{ij}$ are contained in the plane $l_i \vee l_{ij}$ and hence $Q^{ij}$ is
  planar. The planarity of $Q^{jk}$ follows from Lemma~\ref{lem:fundamental_cubes}
  under the assumption of the proposition.
\end{proof}

Laplace transforms of generic Q-nets are Q-nets themselves. This follows from
Proposition~\ref{prop:fundamental_line_systems} since the $i$-th focal net of the $i$-th
edge system is the original Q-net. Hence the following

\begin{corollary}
  Edge systems of generic Q-nets are fundamental.
\end{corollary}

\begin{definition}
\label{def:extension_inscribed}
Let $x:\Z^m \to \RP^N$ be a quadrilateral net and $l:\Z^m \to \linespace^N$ be a
discrete line system. If $x(z) \in l(z)$ for all $z \in \Z^m$ we say that $l$ is
an \emph{extension} of $x$ and, conversely, that $x$ is \emph{inscribed} in $l$.
\end{definition}

Throughout this article, the extension of Q-nets to discrete torsal line systems
as well as the construction of Q-nets inscribed in discrete torsal lines systems
plays a prominent role (line systems and Q-nets related as in
Definition~\ref{def:extension_inscribed} are called \emph{conjugate}
in~\cite{DoliwaSantiniManas:2000:TransformationsOfQnets}). In fact, it turns out
that a discrete torsal line system is fundamental if and only if it is the
extension of a Q-net.  For a concise treatment of this relation and for later
reference, we formulate the following evident

\pagebreak

\begin{lemma}[Construction schemes]
\label{lem:2d_systems}
\hfill
\begin{enumerate}[\qquad]
\item[(Q)]
  Construction of a planar quadrilateral inscribed in a torsal line congruence
  quadrilateral $(l,l_1,l_{12},l_2)$: Three given points $x \in l$, $x_1 \in l_1$, and
  $x_2 \in l_2$ determine the fourth point $x_{12} = l_{12} \cap (x \vee x_1 \vee x_2)$
  uniquely.
\item[(L)]
  Extension of a planar quadrilateral $(x,x_1,x_{12},x_2)$ to a torsal line congruence
  quadrilateral: Three given lines $l \ni x$, $l_1 \ni x_1$, and $l_2 \ni x_2$ determine
  the fourth line $l_{12} = (x_{12} \vee l_1) \cap (x_{12} \vee l_2)$ uniquely.
\end{enumerate}
Each of the above constructions (Q) and (L) describes a 2D system in the sense
that 1-dimensional Cauchy data along two intersecting coordinate axes of $\Z^2$
propagates uniquely onto the entire lattice.
\end{lemma}

Let $l$ be a generic discrete torsal line system in~$\RP^N$ for $N \ge 4$, which
is automatically fundamental due to Theorem~\ref{thm:focal_nets_are_qnets}.  Obviously
the central projection of $l$ from a point to a 3-dimensional subspace preserves
the planarity of focal quadrilaterals and hence yields a fundamental line system
in $\RP^3$.  It turns out that the converse is also true.

\begin{theorem}[Characterization of fundamental line systems in $\RP^3$]
\label{thm:fundamental_line_complex}
For a discrete torsal line system $l : \Z^m \to \linespace^3$,
the following properties are equivalent:
\begin{enumerate}[(i)]
\item
\label{item:thm_fundamental_line_complex_fundamental}
$l$ is fundamental.
\item
\label{item:thm_fundamental_line_complex_projection}
$l$ is the projection of a generic discrete torsal line system in $\RP^4$.
\item
\label{item:thm_fundamental_line_complex_one_inscribed_qnet}
$l$ is the extension of a generic Q-net.
\item
\label{item:thm_fundamental_line_complex_consistency}
The construction of Q-nets inscribed into $l$ is consistent.
\end{enumerate}
\end{theorem}

\begin{proof}
  \iref{item:thm_fundamental_line_complex_fundamental}~$\implies$~\iref{item:thm_fundamental_line_complex_projection}. 
	We start with the observation that for each elementary 3-cube of a fundamental
	line system $l$ in $\RP^3$, any seven lines determine the eighth line uniquely
	via the planarity condition for focal quadrilaterals. For the following, we
	embedd $l$ into $\RP^4$ by identification of $\RP^3$ with any hyperplane $\pi
	\subset \RP^4$ and in a first step show that each individual cube may be
	derived via projection.  More precisely, for its eigth lines
	$l,\ldots,l_{123}$ in $\pi$ we show how to construct lines $\hat l,\ldots,\hat
	l_{123}$ in $\RP^4$ that constitute a fundamental line complex cube, such that
	$l = \tau(\hat l),\ldots,l_{123}=\tau(\hat l_{123})$ with $\tau$ being the
	projection of $\RP^4$ from a point $p \notin \pi$ to the hyperplane $\pi$.  To
	begin with, we may choose $p$ to be any point not in $\pi$. Further, we define
	$\hat{l}=l$, $\hat{l}_1=l_1$, $\hat{l}_{2}=l_2$, and $\hat{l}_{12}=l_{12}$
	in~$\pi$. The lift~$\hat{l}_3$ may then be chosen to be any line in the plane
	$l_3 \vee p$ that passes through $l_3 \cap l$ and lies outside~$\pi$. This
	determines the lifts~$\hat{l}_{13}$ and~$\hat{l}_{23}$ as each of them is
	spanned by one focal point in $\pi$ and one lifted focal point on $\hat l_3$,
	e.g., $\hat l_{13} = (l_{13} \cap l_1) \vee (\hat{l}_3 \cap ((l_{13} \cap l_3)
	\vee p))$.  The seven lifted lines constitute generic Cauchy data for a
	fundamental line complex cube in~$\RP^4$ and hence there exists a unique
	eighth line $\hat l_{123}$ that may be constructed as
  \[
  \hat{l}_{123} = (\hat{l}_{12} \vee \hat{l}_{13}) \cap (\hat{l}_{12}
  \vee \hat{l}_{23}) \cap (\hat{l}_{13} \vee \hat{l}_{23}).
  \]
	The central projection $\tau : \RP^4 \setminus \left\{ p \right\} \to \pi$
	then yields a fundamental line complex cube in~$\pi$.  As seven lines coincide
	with the original cube by construction, the eighth line has to coincide by
	uniqueness.  It remains to observe that the construction works also globally,
	since the lift is generic and thus the suggested construction may be extended
	to Cauchy data for the whole fundamental line system $l$.

	\iref{item:thm_fundamental_line_complex_projection}~$\implies$~\iref{item:thm_fundamental_line_complex_consistency}.
	We have to show the 3D consistency of the 2D system that is described by
	Lemma~\ref{lem:2d_systems}~(Q) under the given assumptions.
	So consider an elementary cube of $l$ and let $x, x_1, x_2, x_3$ be four
	generic points on the lines $l, l_1, l_2, l_3$, respectively. As before,
	consider a lift~$\hat{l}$ of~$l$ to~$\RP^4$ and accordingly lifted points
	$\hat{x}, \hat{x}_1, \hat{x}_2, \hat{x}_3$, the latter spanning a 3-dimensional
	subspace~$\hat{x} \vee \hat{x}_1 \vee \hat{x}_2 \vee \hat{x}_3$. This
	subspace defines a unique Q-cube inscribed into $\hat l$ via intersection
	with the lines of~$\hat{l}$. The projection of the Q-cube in $\RP^4$
	yields a Q-cube in~$\RP^3$, that is, the 2D system is 3D, and hence $m$D
	consistent.
    
  \iref{item:thm_fundamental_line_complex_consistency} $\implies$
  \iref{item:thm_fundamental_line_complex_one_inscribed_qnet}. 
  If the construction of inscribed Q-nets is consistent, then we can
  construct an inscribed Q-net.
  
  \iref{item:thm_fundamental_line_complex_one_inscribed_qnet}
  $\implies$ \iref{item:thm_fundamental_line_complex_fundamental}. 
	It is sufficient to show that each elementary 3-cube of $l$ may be obtained as
	the	projection of a fundamental line complex cube in $\RP^4$, since this implies
	that all focal nets are Q-nets. So assume you
	are given a torsal line complex cube $(l,\ldots,l_{123})$ with an inscribed
	Q-cube~$X=(x,\ldots,x_{123})$ in $\RP^3$. The first thing to note is that the
	vertices $X$ together with the lines $l,l_1,l_2$, and $l_3$ determines the
	remaining lines uniquely according to Lemma~\ref{lem:2d_systems}~(L). Now identify
	$\RP^3$ with any hyperplane $\pi$ in $\RP^4$ and choose a point $p$ outside
	$\pi$, which defines the central projection $\tau : \RP^4 \setminus \left\{ p
	\right\} \to \pi$. The next step is to fix the vertices $X$ in $\pi$ and
	construct generic lifts $\hat{l}$, $\hat{l}_1$, $\hat{l}_2$, $\hat{l}_3
	\not\subset \pi$ of the lines~$l$, $l_1$, $l_2$, and $l_3$ through the
	corresponding vertices that satisfy
	\begin{equation}
	\label{eq:cauchy_projection}
	\hat l \cap \hat l_i \ne \emptyset,
	\quad
	\tau(\hat l)=l,
	\quad
	\tau(\hat l_i)=l_i,
	\quad
	i = 1,2,3.
	\end{equation}
	We will now show that Lemma~\ref{lem:2d_systems}~(L) consistently determines lines
	$\hat l_{12}, \hat l_{23}, \hat l_{13}$, and $\hat l_{123}$ through the
	corresponding vertices of $X$ for the given Cauchy data $\hat l$, $\hat l_1$,
	$\hat l_2$, $\hat l_3$. This implies the assertion, since the solution $(\hat
	l,\ldots,\hat l_{123})$ is mapped under $\tau$ to the solution
	$(l,\ldots,l_{123})$ of the original Cauchy problem because of
	\eqref{eq:cauchy_projection} and $\tau(X) = X$.

	To see the consistency in $\RP^4$, first construct the lines $\hat{l}_{ij}$
	according to Lemma~\ref{lem:2d_systems}~(L). As we considered a generic lift, the
	intersection
	\[
  \hat{l}_{123} = (\hat{l}_{12} \vee \hat{l}_{13}) \cap (\hat{l}_{12}
  \vee \hat{l}_{23}) \cap (\hat{l}_{13} \vee \hat{l}_{23})
  \]
	of three generic hyperplanes in $\RP^4$ is 1-dimensional and completes a fundamental line complex
	cube in $\RP^4$. By construction, we also find $x_{123} \in \hat l_{123}$.
\end{proof}

The above proof of Theorem~\ref{thm:fundamental_line_complex} reveals the following

\begin{corollary}[$m$D consistent systems related to fundamental line systems]
\label{cor:2d_quad_line_systems}\hfill
\begin{enumerate}[(i)]
\item
Fundamental line systems are governed by an $m$D consistent 3D system.
\item
The construction of Q-nets inscribed in a discrete torsal line system according
to Lemma~\ref{lem:2d_systems}~(Q) is $m$D consistent if and only if the line system
is fundamental.
\item
The extension of Q-nets to discrete torsal line systems according to
Lemma~\ref{lem:2d_systems}~(L) is $m$D  consistent and always yields fundamental line
systems.
\end{enumerate}
\end{corollary}

\section{Supercyclides}
\label{sec:supercyclides}

Let $\mathcal{M}$ be a surface in $\RP^3$ that is generated by a 1-parameter
family of conics. If along each conic the tangent planes to $\mathcal{M}$
envelop a quadratic cone, $\mathcal{M}$ is called a \emph{surface of Blutel},
cf. Fig.~\ref{fig:supercyclides}, left.  If the conjugate curves to the conics of
a surface of Blutel are also conics, the Blutel property is automatically
satisfied for the conjugate family \cite{Degen:1982:ConjugateConics}.
Accordingly, such surfaces had been introduced by Degen as \emph{double Blutel
surfaces} originally, but eventually became known to a wider audience under the
name of \emph{supercyclides}.\footnote{Most double Blutel surfaces are (complex)
projective images of Dupin cyclides \cite{Degen:1986:ZweifachBlutel}, so Degen
later referred to double Blutel surfaces as ``generalized cyclides''
\cite{Degen:1994:GeneralizedCyclidesCAGD} -- although this term was already used
by Casey and Darboux for quartic surfaces that have the imaginary circle at
infinity as a singular curve and hence generalize Dupin cyclides in a different
way.  Therefore, Pratt proposed the name ``supercyclides'' for a major subclass
of double Blutel surfaces, which is characterized by a certain quartic equation
and contains the projective images of quartic Dupin cyclides
\cite{Pratt:1996:DupinCyclidesAndSupercyclides,Pratt:1997:QuarticSupercyclidesBasic}.
Eventually, the term supercyclides was used for the whole class of
double Blutel surfaces by Pratt and Degen
\cite{Pratt:1997:SupercyclidesClassification,Degen:1998:SupercyclidesOrigin,Pratt:2002:QuarticSupercyclidesDesign}.}

\begin{figure}[htb]
\begin{center}
\includegraphics[width=.21\textwidth]{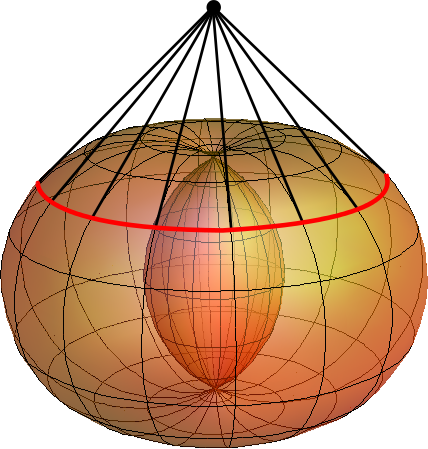}
\hspace{.15\textwidth}
\includegraphics[width=.24\textwidth]{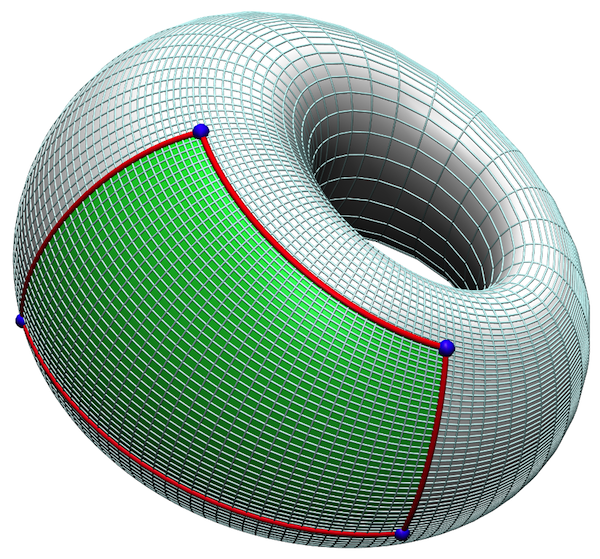}
\end{center}
\caption{Left: The tangent planes along a characteristic conic on a supercyclide
  envelop a quadratic cone. The generators of the cone are the tangents of the
  conjugate curves.
  Right: Restriction of a supercyclide to an SC-patch, that is, a surface patch
  bounded by parameter lines of a characteristic parametrization.}
\label{fig:supercyclides}
\end{figure}

\begin{definition}[Supercyclides and SC-patches]
\label{def:supercyclide}
A surface in $\RP^3$ that is generated by two conjugate families of conics, such
that the tangent planes along each conic envelop a quadratic cone, is called a
\emph{supercyclide}.  We refer to those conics as \emph{characteristic conics}
and accordingly call a conjugate parametrization $f : U \to \RP^3$ of a
supercyclide characteristic if its parameter lines are characteristic conics.

A (parametrized) \emph{supercyclidic patch} (\emph{SC-patch} for short) is the
restriction of a characteristic parametrization to a closed rectangle $I_1
\times I_2 \subset U$, cf. Fig.~\ref{fig:supercyclides}, right.
\end{definition}

In~\cite{Degen:1982:ConjugateConics} it is shown that
for each of the two families of characteristic conics on a supercyclide the
supporting planes form a pencil. The axis of such a pencil consists of the
vertices of the cones that are tangent to the respective conjugate family of
characteristic conics.

\begin{definition}[Characteristic lines]
\label{def:characteristic_lines}
The axes of the pencils supporting the families of characteristic conics of a
supercyclide are the \emph{characteristic lines}.
\end{definition}

\paragraph{Examples of supercyclides}
Any non-degenerate quadric $Q$ in $\RP^3$ is a supercyclide in a manifold way:
Simply take two lines $a^1$ and $a^2$ that are polar with respect to $Q$ and
consider the plane pencils through those lines. The intersections of those
planes with $Q$ consitute a conjugate network of conics that satisfy the tangent
cone property. Another prominent subclass of supercyclides is given by Dupin
cyclides, the latter being characterized by the fact that all curvature lines
are circles. They are supercyclides, since a net of curvature lines is a special
conjugate net and, moreover, the tangent cone property is satisfied. As special
instances of Dupin cyclides it is worth to mention the three types of
rotationally symmetric tori (ring, horn, and spindle -- having 0, 1, and 2
singular points, respectively) and to recall that all other Dupin cyclides may
be obtained from such tori by inversion in a sphere. Finally, since the
defining properties of supercyclides are projectively invariant, it is clear
that any projective image of a Dupin cyclide is a supercyclide.

\smallskip

A complete classification of supercyclides is tedious, see the early
contributions
\mbox{\cite{Degen:1986:ZweifachBlutel,Barner:1987:DifferentialgeometrischeKennzeichnungSuperzykliden}}
and also the later
\cite{Pratt:1997:QuarticSupercyclidesBasic,Pratt:1997:SupercyclidesClassification}.
One fundamental result is that supercyclides are algebraic surfaces of, at most,
degree four.  While
\cite{Degen:1986:ZweifachBlutel,Barner:1987:DifferentialgeometrischeKennzeichnungSuperzykliden}
follow a classical differential geometric approach that relies essentially on
convenient parametrizations derived from the double Blutel property, the
treatment of supercyclides in
\cite{Pratt:1997:QuarticSupercyclidesBasic,Pratt:1997:SupercyclidesClassification}
is more of an algebraic nature. On the other hand, there is also a unified
approach to the construction and classification of supercyclides that starts
with a certain ruled 3-manifold in projective 5-space from which all
supercyclides may be obtained by projection and intersection with a hyperplane
\cite{Degen:1998:SupercyclidesOrigin}. The basis for this unified treatment can
be found already in \cite{Degen:1986:ZweifachBlutel}.

\paragraph{Genericity assumption}
In this paper we only consider generic supercyclides, that is, supercyclides of
degree four with skew characteristic lines.

\enlargethispage{\baselineskip}

\begin{proposition}[Properties of SC-patches]
\label{prop:properties_supercyclidic_patches}
\hfill
\begin{enumerate}[(i)]
\item
\label{item:prop_patches_coplanar}
The vertices of an SC-patch are coplanar.
\item
\label{item:prop_patches_projection}
Isoparametric points on opposite boundary curves are perspective from the
non-cor\-re\-spond\-ing Laplace point of the vertex quadrilateral, cf.
Fig.~\ref{fig:scpatches_projection}, left.
\item
\label{item:prop_patches_concurrent_tangent_congruence_quad}
For each coordinate direction, the tangents to the corresponding boundary curves
at the four vertices of an SC-patch intersect cyclically.  In particular, the
eight tangents constitute an elementary hexahedron of a fundamental line
complex, cf.  Fig.~\ref{fig:scpatches_projection}, middle and right.
\item
\label{item:prop_patches_concurrent_tangent_planes}
The tangent planes at the four vertices intersect in one point.
\end{enumerate}
\end{proposition}

\begin{figure}[htp]
\begin{center}
\parbox{.35\textwidth}{\includegraphics[width=.35\textwidth]{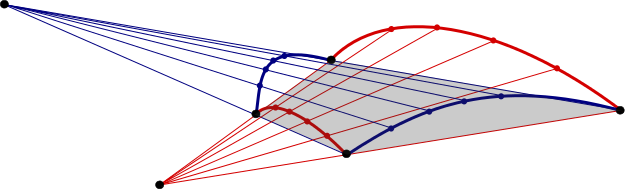}}
\qquad
\qquad
\parbox{.5\textwidth}{ 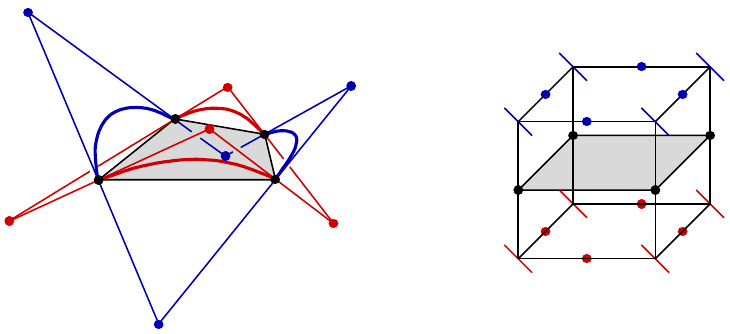 }
\end{center}
\caption{Opposite boundary curves of an SC-patch are perspective from the
non-corresponding Laplace point of the vertex quadrilateral (left). The
tangents to the boundary of an SC-patch at vertices constitute an elementary
hexahedron of a fundamental line complex (middle and right: geometric and
combinatorial, respectively).}
\label{fig:scpatches_projection}
\end{figure}

\begin{proof}
We start with a fact concerning Blutel surfaces that may be found in
\cite{Degen:1986:ZweifachBlutel}, that is, for any surface of Blutel the
conjugate curves induce projective transformations between the generating
conics. From that it may be deduced easily that for a double Blutel surface
those projective maps are perspectivities. Applied to opposite
boundary curves of an SC-patch, we conclude that the patch vertices are coplanar
and that the centers of perspectivity have to be the Laplace points of the
vertex quadrilateral.

To verify the perspectivity statement, consider two characteristic conics $c_1$
and $c_2$ of one family on a supercyclide for which we assume that they are
contained in distinct planes $\pi_1$ and $\pi_2$.  According to the above, the
conjugate family of conics induces a projective map $\tau : \pi_1 \to \pi_2$
that identifies isoparametric points on $c_1$ and $c_2$. In order to see that
$\tau$ is a perspectivity, take three points $x_1,y_1,z_1$ on $c_1$ and denote
the tangents in those points by $g_1,h_1,l_1$, respectively. Further let
$(x_2,y_2,z_2,g_2,h_2,l_2) = \tau (x_1,y_1,z_1,g_1,h_1,l_1)$ be the
corresponding points and tangents with respect to $c_2$.  As, e.g., $x_1$ and
$x_2$ are isoparametric, the tangents $g_1$ and $g_2$ intersect at the tip $x$
of the cone that is tangent to the supercyclide along the conjugate curve
through $x_1$ and $x_2$. One obtains three tangent cone vertices $x,y,z$ of that
kind and also the further intersection points $q_i = g_i \cap h_i$, $r_i = g_i
\cap l_i$, $s_i = h_i \cap l_i$, $i=1,2$, see Fig.~\ref{fig:perspectivity_proof}.

\begin{figure}[htb]
\begin{center}
 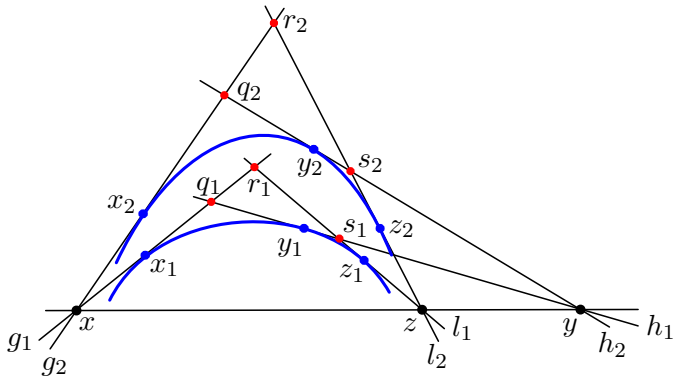 
\end{center}
\caption{Triangles formed by the tangents to $c_1$ and $c_2$ in corresponding
points $x_1,y_1,z_1$ and $x_2,y_2,z_2$, are perspective.}
\label{fig:perspectivity_proof}
\end{figure}

From the observation $\tau(x,y,z,q_1,r_1,s_1) = (x,y,z,q_2,r_2,s_2)$ it follows
that the restrictions of $\tau$ to the lines $g_1$, $h_1$, and $l_1$, all being
projective maps, are determined by $(x,q_1,r_1) \mapsto (x,q_2,r_2)$, $(x,q_1,s_1)
\mapsto (x,q_2,s_2)$, and $(x,s_1,r_1) \mapsto (x,s_2,r_2)$, respectively.  On
the other hand, the three planes $g_1 \vee g_2, h_1 \vee h_2$, and $l_1 \vee
l_2$ intersect in one point $O$, therefore the pairs $(q_1,q_2), (r_1,r_2)$, and
$(s_1,s_2)$ of corresponding tangent intersections are each perspective from
$O$. From this we may conclude that $\tau : \pi_1 \to \pi_2$ is the perspectivity from $O$.

It remains to verify the assertions
\iref{item:prop_patches_concurrent_tangent_congruence_quad} and
\iref{item:prop_patches_concurrent_tangent_planes}. Now the intersection
statement of \iref{item:prop_patches_concurrent_tangent_congruence_quad} follows
immediately from the tangent cone property and planarity of the characteristic
conics and simply means that the eight tangents to the boundary curves at
vertices constitute a line complex cube. The vertex quadrilateral of the patch
is a focal quadrilateral for that complex cube and planarity of vertices implies
that the complex cube is fundamental due to Lemma~\ref{lem:fundamental_cubes}~\iref{item:lem_fun_cubes_planar_quad_one}.  Therefore, the equivalent statement
\iref{item:lem_fun_cubes_concurrent_planes_any} of Lemma~\ref{lem:fundamental_cubes}
also holds and yields \iref{item:prop_patches_concurrent_tangent_planes} of this
proposition.
\end{proof}

\begin{remark}
\label{rem:singularities}
In the complexified setting, a generic quartic supercyclide possesses four
isolated singularities -- two at each characteristic line -- and one singular conic
\cite{Pratt:1997:QuarticSupercyclidesBasic}. In fact, it can be easily seen that
real intersections of a supercyclide with its characteristic lines must be
singularities: Suppose that the characteristic line of the first family (the
axis of the pencil of planes supporting the first family of conics) intersects
the cyclide in a point~$x$. Generically, $x$ lies on a non-degenerate
conic of the first familiy and
Proposition~\ref{prop:properties_supercyclidic_patches}~\iref{item:prop_patches_projection}
implies that all conics of the first family
pass through $x$, that is, $x$ is a singular point.  On the
other hand, if a generic pair of conics of one family intersects (automatically
at the corresponding characteristic line), the intersection points are singular
points of the cyclide by the same argument.
\end{remark}

\paragraph{Notation and terminology for SC-patches}
We denote the vertices of an SC-patch $f:[a_0,a_1] \times [b_0,b_1]
\to \RP^3$ by
\begin{equation*}
x = f(a_0,b_0),\
x_1 = f(a_1,b_0),\
x_{12} = f(a_1,b_1),\
x_2 = f(a_0,b_1),
\end{equation*}
which gives rise to the (oriented) \emph{vertex quadrilateral}
$(x,x_1,x_{12},x_2)$. Conversely, we say that an SC-patch is \emph{adapted} to a
quadrilateral if it is the vertex quadrilateral of the patch.
We define
\begin{equation*}
\arcs = \left\{ \text{conic arcs in } \RP^3 \right\}
\end{equation*}
and write the boundary curve of an SC-patch that connects the vertices $x$ and
$x_i$ as $b^i \in \arcs$. The supporting plane of $b^i$ in $\RP^3$ is usually
denoted by $\pi^i$ and we refer to those planes as \emph{boundary planes} of the
patch.  Further, we denote by $t^i$ and $t^i_i$ the tangents to $b^i$ at $x$ and
at $x_i$, respectively, as depicted in Fig.~\ref{fig:patch_notation}.

\begin{figure}[htb]
\begin{center}
 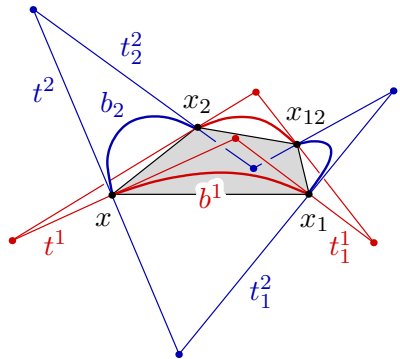 
\end{center}
\caption{The boundary of an SC-patch, labeled according to our usual terminology.}
\label{fig:patch_notation}
\end{figure}

Usually we refer to the characteristic lines of an SC-patch as $a^1$ and $a^2$,
where $a^i$ is the line through the Laplace point
\begin{equation*}
y^i = (x \vee x_i) \cap (x_j \vee x_{12}),\quad i \ne j.
\end{equation*}
Obviously, the characteristic lines may be expressed in
terms of tangents as
\begin{equation*}
a^i = (t^i \vee t^i_i) \cap (t^i_j \vee t^i_{12}),\quad i \ne j.
\end{equation*}

\paragraph{Extension of planar quadrilaterals to SC-patches}
A supercyclidic patch is uniquely determined by its boundary, since points $\tilde
x_i \in b^i$, $i = 1,2$, on the boundary of an SC-patch parametrize points
$\tilde x_{12}$ in the interior as follows: Given the characteristic lines $a^1$
and $a^2$, the corresponding point $\tilde x_{12}$ may be obtained as
intersection of the three planes $(x \vee \tilde x_1 \vee \tilde x_2)$, $(\tilde
x_1 \vee a^2)$, and $(\tilde x_2 \vee a^1)$.

According to Proposition~\ref{prop:properties_supercyclidic_patches}, valid boundaries of
SC-patches that are adapted to a given vertex quadrilateral may be constructed
as follows: Two adjacent conical boundary curves may be chosen arbitrarily as
well as the two opposite boundary planes.  The remaining boundary curves are
then determined as perspective images of the initial boundary curves from the
corresponding Laplace points of the vertex quadrilateral. It turns out that it
is very convenient to emphasize the tangents to boundary curves at vertices in
the suggested construction, which gives rise to the following

\paragraph{Construction of adapted SC-patches}
Denote the given (planar) vertex quadrilateral by $(x,x_1,x_{12},x_2)$. For each
$i = 1,2$, first choose $t^i \ni x$ freely and then choose $t^i_1 \ni x_1$ and
$t^i_2 \ni x_2$ such that $t^i \cap t^i_j \ne \emptyset$. According to
Proposition~\ref{prop:properties_supercyclidic_patches}~\iref{item:prop_patches_concurrent_tangent_congruence_quad},
the tangent $t^i_{12}$ is then determined subject to $t^i_{12} = (x_{12} \vee
t^i_1) \cap (x_{12} \vee t^i_2)$. Further, for each coordinate direction choose
a conic segment $b^i$ in the plane $\pi^i = t^i \vee t^i_i$ that is tangent to
$t^i$ at $x$ and to $t^i_i$ at $x_i$, respectively. The opposite segment $b^i_j$
is then obtained as the perspective image $\tau^j(b^i)$, where
\begin{equation*}
\tau^j : \pi^i \to \pi^i_j,
\quad
p \mapsto (p \vee y^j) \cap \pi^i_j
\end{equation*}
is the perspectivity from the Laplace point $y^j = e^j \cap e^j_i$.
\begin{remark}
\label{rem:patch_construction_dof}
In the above construction, we have exploited 10 independent degrees of freedom:
For each direction $i$, 4 DOF are used in the construction of the corresponding
tangents and 1 DOF is used for the adapted boundary curve $b^i$. The latter
corresponds to the free choice of one additional point in the prescribed
boundary plane, modulo points on the same segment.\footnote{Recall that a conic
is uniquely determined by 5 points, where two infinitesimally close points
correspond to a point and a tangent in that point (a 1-dimensional
contact element).}
\end{remark}

\section{2D supercyclidic nets}
\label{sec:2d_supercyclidic_nets}

We will now analyze the extension of a given Q-net $x:\Z^2 \to \RP^3$ by one
adapted supercyclidic patch for each elementary quadrilateral. At the common
edge of two neighboring quadrilaterals we require that the tangent cones
coincide.

\begin{definition}[Tangent cone continuity]
\label{def:tcc_property}
Let $Q$ and $\tilde Q$ be two planar quadrilaterals that share an edge
$e$.  Two adapted SC-patches $f$ and $\tilde f$ with common boundary
curve $b$ that joins the vertices of $e$ are said to form a
\emph{tangent cone continuous join}, or \emph{TCC-join} for short, if
the tangency cones to $f$ and $\tilde f$ along $b$ coincide. We also
say that $f$ and $\tilde f$ have the \emph{TCC-property} and,
accordingly, satisfy the \emph{TCC-condition}.
\end{definition}

\begin{figure}[hbt]
\begin{center}
\parbox{.31\textwidth}{\includegraphics[width=.25\textwidth]{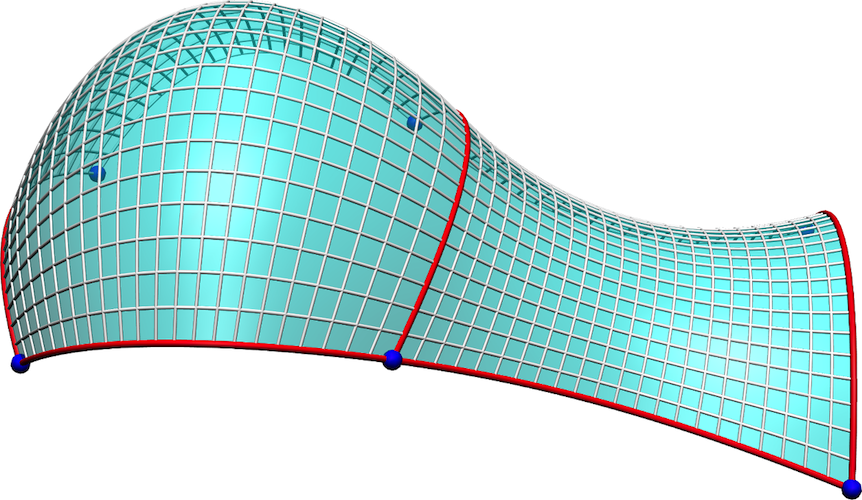}}
\parbox{.31\textwidth}{\includegraphics[width=.25\textwidth]{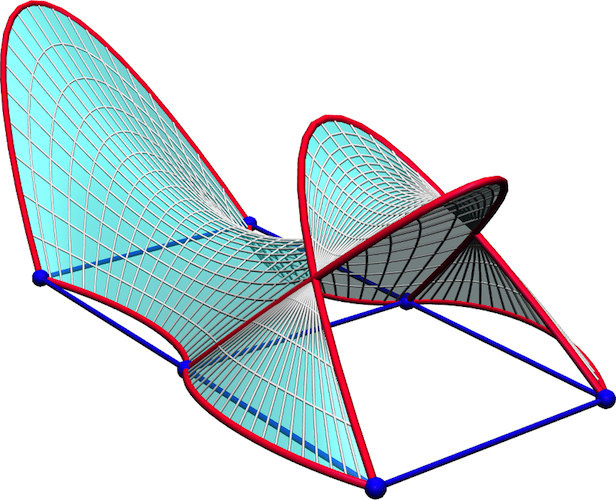}}
\parbox{.28\textwidth}{\includegraphics[width=.28\textwidth]{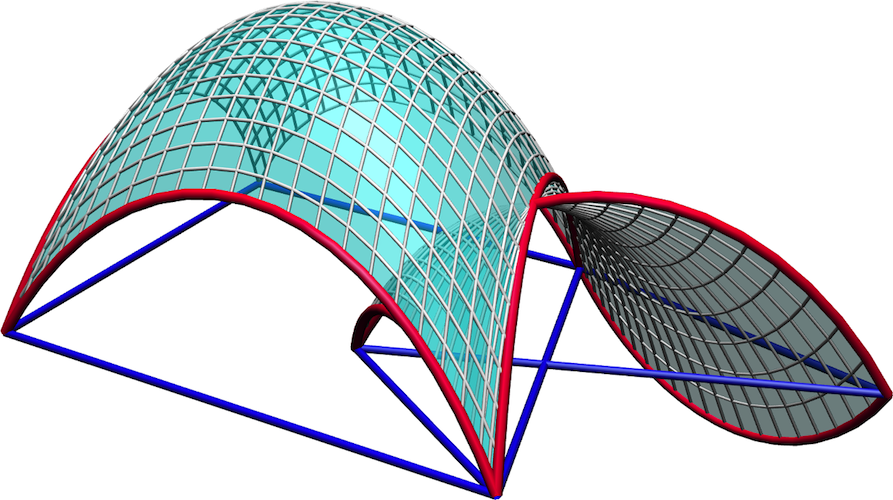}}
\end{center}
\caption{Pairs of SC-patches that form TCC joins, that is, tangency cones
coincide along common boundary curves. As depicted, this does not exclude cusps.}
\label{fig:partial_cusps}
\end{figure}

The TCC-condition yields a consistent theory for the extension of discrete
conjugate nets by surface patches parametrized along conjugate directions
in~$\RP^3$.  As a quadratic cone is determined by one non-degenerate conic
section and two generators, one immediately obtains the following
characterization of TCC-joins between SC-patchs.

\begin{lemma}
\label{lem:tcc_characterization}
Two SC-patches $f,\tilde f$ with common boundary curve $b$ form a TCC-join
if and only if their tangents in the common vertices coincide.
\end{lemma}

Consider coordinates so that $\tilde Q = Q_i$ is the shift of $Q$ in direction
$i$ and accordingly rename $(\tilde f,b) \leadsto (f_i,b^j_i)$.  If the relevant
tangents coincide, they may be expressed as intersections of boundary planes
\begin{equation*}
t^i_i = \pi^i \cap \pi^i_i,
\quad
t^i_{ij} = \pi^i_j \cap \pi^i_{ij},
\end{equation*}
and we see that the TCC-property implies that the four involved boundary
planes of direction $i$ intersect in one point. Obviously, this means that the
two characteristic lines
\begin{equation*}
a^i = \pi^i \cap \pi^i_j,
\quad
a^i_i = \pi^i_i \cap \pi^i_{ij}
\end{equation*}
associated with the direction $i$ intersect -- the point of intersection being
the tip of the coinciding tangency cone.

\begin{definition}[2D supercyclidic net]
\label{def:sc_net_2d_patches}
A 2-dimensional Q-net together with SC-patches that are adapted to its
elementary quadrilaterals is called a \emph{2D supercyclidic net} (2D
\emph{SC-net}) if each pair of patches adapted to edge-adjacent quadrilaterals
has the TCC-property.
\end{definition}

\paragraph{Affine supercyclidic nets and $C^1$-joins}
An interesting subclass of supercyclidic nets are nets that form piecewise
smooth $C^1$-surfaces, i.e., the joins of adjacent patches are not allowed to
form cusps and the individual patches must not contain singularities (see
Fig.~\ref{fig:partial_cusps}). Concise statements concerning the $C^1$-subclass of
tangent cone continuous supercyclidic patchworks are tedious to formulate and to
prove. As a starting point for future investigation of the $C^1$-subclass we
only capture some observations in the following.

A Q-net $x : \Z^2 \to \RP^3$ does not contain information about edges connecting
adjacent vertices. In the affine setting in turn one has unique finite edges,
which allow to distinguish between the different types of non-degenerate vertex
quadrilaterals> convex, non-convex, and non-embedded. We note the following
implications for adapted SC-patches:

\begin{enumerate}[(i)]
\item Finite SC-patches adapted to \emph{non-embedded} vertex quadrilaterals
  always contain singular points: The intersection point of one pair of edges is
  the Laplace point and it lies on the characteristic line of the pencil
  containing the conic arcs. But if these arcs are finite they have to intersect
  the axis and hence contain a singular point of the surface, see
  Remark~\ref{rem:singularities} and Fig.~\ref{fig:affine_quads}, left.
\item If the vertex quadrilateral is \emph{non-convex} an adapted SC-patch
	cannot be finite: For any SC-patch, opposite boundary curves are perspective
	and hence lie on a unique cone whose vertex is one of the Laplace points of
	the vertex quadrilateral. For non-convex vertex quadrilaterals, the Laplace
	points are always contained in edges of the quadrilateral. Thus, the
	previously mentioned cones always give rise to a situation as depicted in
	Fig.~\ref{fig:affine_quads}, right.  We observe that one of the edges $(x,x_2)$ or
	$(x_1, x_{12})$ has to lie outside of the cone and the other one inside.
	Clearly, the conic boundary arc that connects the vertices of the outside
	segment has to be unbounded.
\item If the vertex quadrilateral is \emph{convex} (and embedded) then there
  exist finite patches without singularities, but adapted patches have
  singularities if opposite boundary arcs intersect.
\end{enumerate}

\begin{figure}[htb]
  \parbox{0.47\textwidth}{
	 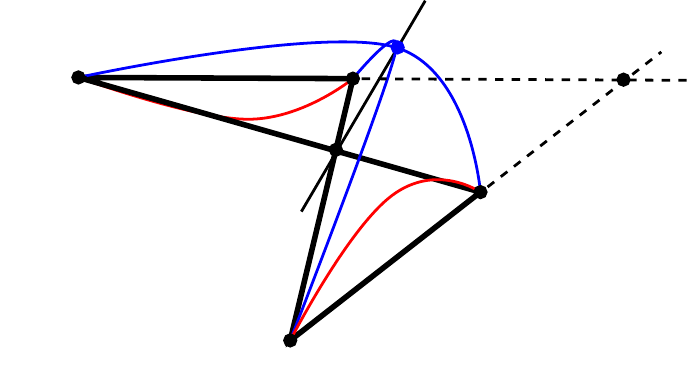 
	}
  \hspace{2em}
  \parbox{0.4\textwidth}{
	 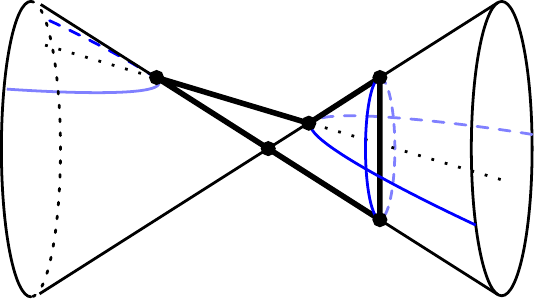 
	}
	\caption{SC-patches adapted to different types of affine quadrilaterals. Left:
	For finite patches adapted to non-convex quadrilaterals, two opposite arcs
	have to intersect. Right: Different shades show the possible pairs of
	perspective conic arcs for a non-convex vertex quadrilateral related by the
	perspectivity with respect to the Laplace point~$y^1$.}
  \label{fig:affine_quads}
\end{figure}

Accordingly, only affine Q-nets composed of convex quadrilaterals may be
extended to bounded $C^1$-surfaces (projective images thereof also being $C^1$,
but possibly not bounded). In that context, note that the cyclidic nets of
\cite{BobenkoHuhnen-Venedey:2011:cyclidicNets} are defined as extensions of
circular nets with embedded quadrilaterals (automatically convex) in order to
avoid singularities.

\paragraph{Discrete data for a 2D supercyclidic net}
We have seen that supercyclidic patches are completely encoded in their boundary
curves and that the TCC-property may be reduced to coinciding tangents in common
vertices. Accordingly, a 2-dimensional supercyclidic net is completely encoded
in the following data:

\begin{enumerate}[(i)]
\item The supporting Q-net
$x: \Z^2 \to \RP^3$.
\item The congruences of tangents
$t^1,t^2 : \Z^2 \to \linespace^3$.
\item The boundary curves
$b^1,b^2 : \Z^2 \to \arcs$ with
$b^i$ tangent to $t^i$ at $x$ and to $t^i_i$ at $x_i$
and such that for each quadrilateral opposite
boundary curves are perspective from the corresponding Laplace point.
\end{enumerate}

\begin{remark}
\label{rem:discrete_data_2d}
In order to obtain discrete data that consists of points and lines only, for
given Q-net $x$ and tangents $t^i$ the conic segment $b^i$ may be represented by
one additional point in its supporting plane $\pi^i$ (see also
Remark~\ref{rem:patch_construction_dof}).  Accordingly, the boundary
curves are then described by maps $b^i : \Z^2 \to \RP^3$ with $b^i \in t^i \vee
t^i_i$. This representation does not only fix the conic arc but also a notion of
parametrized supercyclidic nets.
\end{remark}

The tangents at vertices of a supercyclidic net are obviously determined by the
conic splines composed of the patch boundaries.  For the formulation of a
related Cauchy problem it is more convenient to capture the tangents and
boundary curves separately.

\paragraph{Cauchy data for a 2D supercyclidic net}
\label{par:cauchy_2d}
According to Lemma~\ref{lem:2d_systems}~(L), the tangents of a 2D SC-net are
uniquely determined by the tangents along coordinate axes if the supporting
Q-net is given. On the other hand, if all the tangents are known, we know all
boundary planes and may propagate suitable initial boundary splines according to
the perspectivity property of SC-patches that is captured by
Proposition~\ref{prop:properties_supercyclidic_patches}~\iref{item:prop_patches_projection}.
Hence Cauchy data of a 2D supercyclidic net consists of, e.g.,

\begin{enumerate}[(i)]
\item
The supporting Q-net $x : \Z^2 \to \RP^3$.
\item 
  The tangents along two intersecting coordinate lines for each coordinate
  direction
\begin{equation*}
t^i|_{\coordsurf{j}},\quad i,j=1,2.
\end{equation*}
\item
Two conic splines
\begin{equation*}
b^i|_{\coordsurf{i}},\quad i=1,2,
\end{equation*}
such that $b^i$ is tangent to $t^i$ at $x$ and to $t^i_i$ at $x_i$,
cf. Fig.~\ref{fig:2d_cyclidic_cauchy}.
\end{enumerate}

\begin{figure}[htb]
\begin{center}
 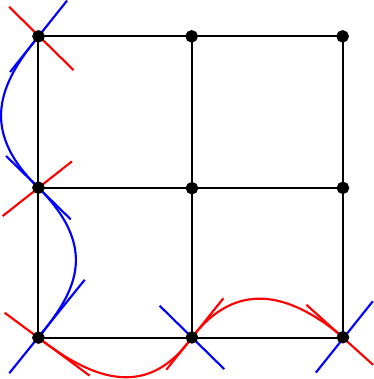 
\end{center}
\caption{Cauchy data for a 2D cyclidic net.}
\label{fig:2d_cyclidic_cauchy}
\end{figure}

\paragraph{Rigidity of supercylidic nets}
An SC-patch is determined by its boundary curves uniquely.  Further, one
observes that for fixed Q-net and fixed tangents the variation of a boundary
curve propagates along a whole quadrilateral strip due to the perspectivity
property.  Moreover, a local variation of tangents is not possible for a
prescribed Q-net, as each tangent is determined uniquely by the contained vertex
of the Q-net and two adjacent tangents. This shows that it is not possible to
vary a supercyclidic net locally. However, we know that local deformation of 2D
Q-nets is possible\footnote{Local deformation of higher dimensional
Q-nets is not possible as they are governed by a 3D system.}
\cite{Hoffmann:2010:PQDeformation}, which gives rise to the question whether
effects of the (global) variation of a given 2D SC-net might be minimized by
incorporation of suitable local deformations of the supporting Q-net.

\pagebreak

\paragraph{Q-refinement induced by 2D supercyclidic nets}
Obviously, any selection of characteristic conics on an SC-patch induces a
Q-refinement of its vertex quadrilateral since also the vertices of the induced
``sub-patches'' are coplanar.  Accordingly, any continuous supercyclidic
patchwork induces an arbitrary Q-refinement of the supporting 2D Q-net.  This is
interesting in the context of the convergence of discrete conjugate nets to
smooth conjugate nets, but also a noteworthy fact from the perspective of
architectural geometry: 2D supercyclidic nets may be realized at arbitrary
precision by flat panels.

\paragraph{Implementation on a computer}
We have implemented two ways to experiment with supercyclidic nets. The first is
in java and solves the Cauchy problem described on page~\pageref{par:cauchy_2d},
using the framework of jReality and
VaryLab.\footnote{\texttt{http://www.jreality.de} and \texttt{http://www.varylab.com}}
We start with a Q-net with $\Z^2$ combinatorics. To create the two tangent
congruences we prescribe initial tangents at the origin and propagate them along
each coordinate axis via reflections in the symmetry planes of the edges (which does not
exploit all possible degrees of freedom). The prescribed tangent congruences on
the coordinate axes can then be completed to tangent congruences on the entire
Q-net according to Lemma~\ref{lem:2d_systems}~(L). The conic splines on the
coordinate axes are described as rational quadratic Bezier curves with
prescribed weight. The initial weights are then propagated to all edges by evaluating
suitable determinant expressions, which take into account the Q-net, the focal nets, the
weights, and the intersection points of the four tangent planes at each
quadrilateral. This yields a complete description of the resulting supercyclidic
net in terms of rational quadratic Bezier surface patches. See
Fig.~\ref{fig:sc_net_implementation} for two different SC-nets adapted to one and
the same supporting Q-net that are obtained that way.  The second approach to
the computation of supercyclidic nets is based on a global variational approach
and uses the framework of Tang et al.~\cite{tang-2014-ff}. It also makes use of
quadratic Bezier surface patches, where quadratic constraints involving
auxiliary variables guarantee that those patches constitute an SC-net.

\begin{figure}[htb]
\begin{center}
\includegraphics[width=.88\textwidth]{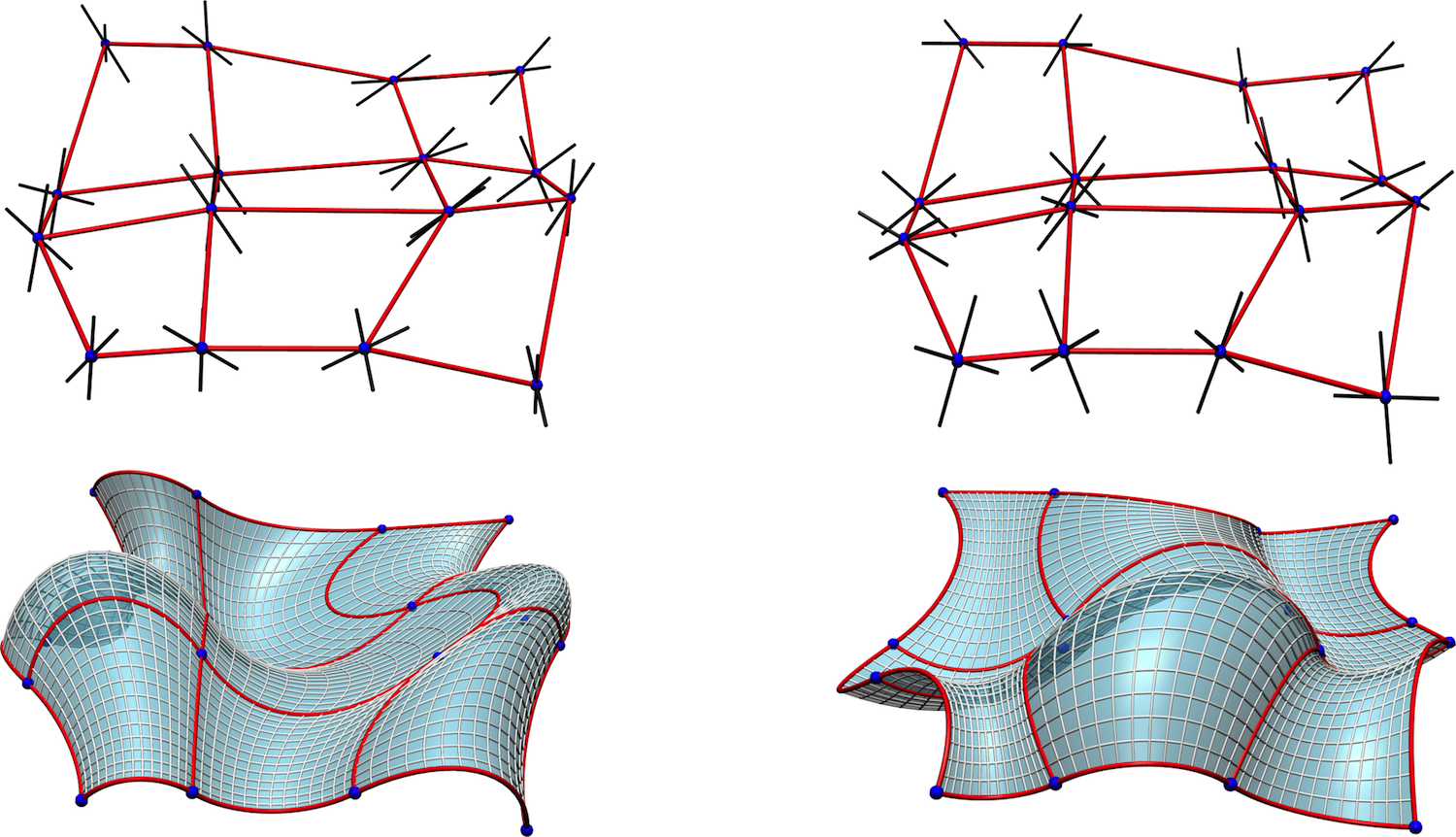}
\end{center}
\caption{Two different supercyclidic nets with same supporting Q-net. The Q-net
with tangents at vertices is drawn on top of the corresponding supercyclidic
net.}
\label{fig:sc_net_implementation}
\end{figure}

\section{3D supercyclidic nets}
\label{sec:3d_supercyclidic_nets}

Motivated by a common structure behind integrable geometries that is reflected
by the ``multidimensional consistency principle'' of discrete differential
geometry, see \cite{BobenkoSuris:2008:DDGBook}, we start out with

\begin{definition}[3D supercyclidic net]
\label{def:sc_net_3d_from_2d}
A Q-net $x:\Z^3 \to \RP^3$ together with SC-patches that are adapted to its
elementary quadrilaterals -- such that patches adapted to edge-adjacent
quadrilaterals meet in a common boundary curve associated with the corresponding
edge -- is called a \emph{3D supercyclidic net} if the restriction to any
coordinate surface is a 2D supercyclidic net.
\end{definition}

While Definition~\ref{def:sc_net_3d_from_2d} appears natural, it is not obvious that one
can consistently (that is, with coinciding boundary curves) adapt SC-patches
even to a single elementary hexahedron of a 3D Q-net.  Before proving that this
is indeed possible (Theorem~\ref{thm:cyclidic_3d_system}), we introduce some convenient
terminology.

\begin{definition}[Common boundary condition and supercyclidic cubes]
Consider several quadrilaterals that share an edge. We say that SC-patches
adapted to those quadrilaterals satisfy the \emph{common boundary condition} if
their boundary curves associated with the common edge all coincide.  Moreover,
we call a 3-dimensional Q-cube with SC-patches adapted to its faces, such that
the common boundary condition is satisfied for each edge, a \emph{supercyclidic
cube} or \emph{SC-cube} for short, cf. Fig.~\ref{fig:sc_cube}.
\end{definition}

\begin{figure}[htb]
\begin{center}
\includegraphics[width=.3\textwidth]{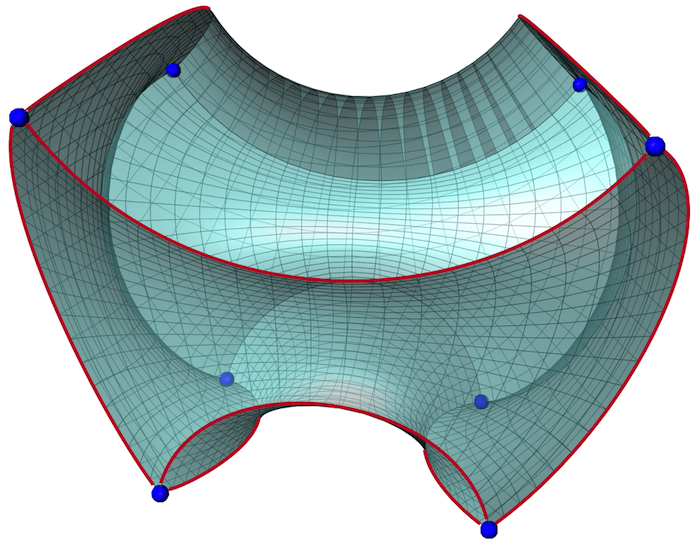}
\hspace{.13\textwidth}
\includegraphics[width=.3\textwidth]{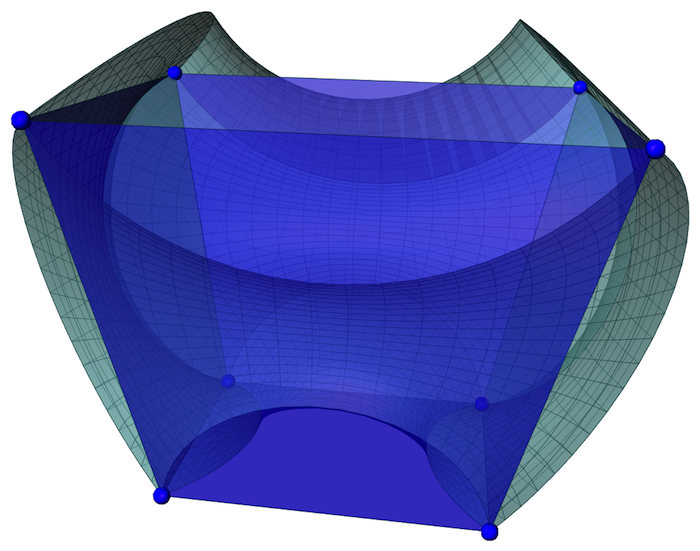}
\end{center}
\caption{A supercyclidic cube: Six SC-patches with shared boundary curves whose
vertices form a cube with planar faces.}
\label{fig:sc_cube}
\end{figure}

\begin{theorem}[Supercyclidic 3D system]
\label{thm:cyclidic_3d_system}
Three faces of an SC-cube that share one vertex, that is, three generic
SC-patches with one common vertex and cyclically sharing a common boundary curve
each, determine the three opposite faces (``SC 3D system''). Accordingly, the
extension of a Q-cube to an SC-cube is uniquely determined by the free choice of
three SC-patches that are adapted to three faces that meet at one vertex and
satisfy the common boundary condition.
\end{theorem}

\begin{proof}
First note that the three given SC-patches determine a unique supporting Q-cube
$C$, since Q-nets are governed by a 3D system and seven vertices of the cube are
already given. Further, denote the vertices of $C$ in
the standard way by $x,\ldots,x_{123}$ so that the vertex quadrilaterals
of the three patches are $Q^{ij} = (x,x_i,x_{ij},x_j)$, $1 \le i < j \le 3$,
and the common vertex is $x$.

Assume that the three given patches may be completed to an SC-cube.  Then
through each edge of $C$ in direction $i$ there is a unique boundary plane that
contains the associated boundary curve. We denote those planes by
\begin{equation*}
\pi^i = t^i \vee t^i_i \supset e^i = x \vee x_i
\end{equation*}
and the respective shifts. Now consider the perspectivities between those
planes that identify the opposite boundary curves of the adapted patches.
They are central projections through the Laplace points (cf.
Fig.~\ref{fig:laplace_projection_cube}, left)
\begin{equation*}
y^{i,j} = (x \vee x_j) \cap (x_i \vee x_{ij}),
\end{equation*}
which we denote by
\begin{equation}
\label{eq:laplace_projection}
\tau^{i,j} : \pi^i \to \pi^i_j,
\quad p \mapsto (p \vee y^{i,j}) \cap \pi^i_j,
\end{equation}
so that
\begin{equation}
\label{eq:boundary_projection_tau}
b^i_j = \tau^{i,j}(b^i).
\end{equation}
Here we make the genericity assumption that the tangents are such, that the
boundary planes do not contain the Laplace points.

Now fix one direction $i$ and consider the corresponding ``maps around the
cube'' as depicted in Fig.~\ref{fig:laplace_projection_cube}, right.

\begin{figure}[htb]
\begin{center}
\qquad
 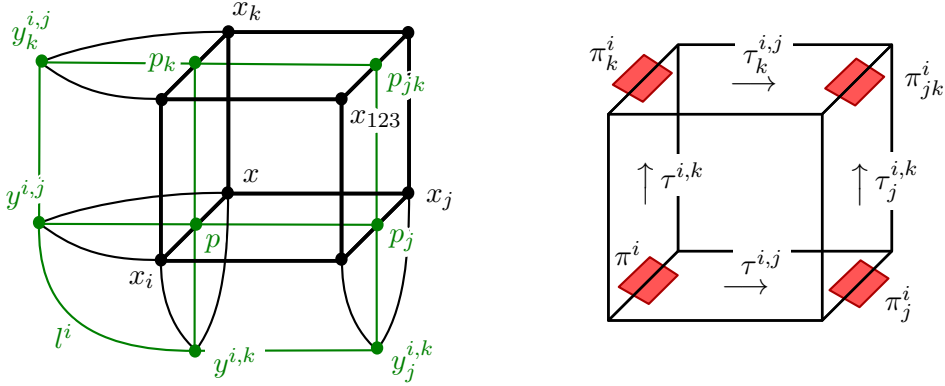 
\end{center}
\caption{Laplace projections around a cube.}
\label{fig:laplace_projection_cube}
\end{figure}

On use of the fact that the four involved Laplace points
$y^{i,j},y^{i,j}_k,y^{i,k},y^{i,k}_j$ are contained in the line
\begin{equation*}
l^i = Q^{jk} \cap Q^{jk}_i,
\end{equation*}
one verifies that the relations
\eqref{eq:boundary_projection_tau} imply
\begin{equation}
\label{eq:commuting_projections}
\tau^{i,j}_k \circ \tau^{i,k} = \tau^{i,k}_j \circ \tau^{i,j}.
\end{equation}

So, given the three patches adapted to $Q^{12},Q^{23}$, and $Q^{13}$, for any
direction $i$ we know already the three planes $\pi^i,\pi^i_j,\pi^i_k$, $\left\{
i,j,k \right\} = \left\{ 1,2,3 \right\}$, and the claim of
Theorem~\ref{thm:cyclidic_3d_system} follows if we show that given those planes there
exists a unique plane $\pi^i_{jk}$ through $e^i_{jk}$ such that
\eqref{eq:commuting_projections} holds.  The fact that $\pi^i_{jk}$ is unique, if
it exists at all, follows again from the collinearity of
$y^{i,j},y^{i,j}_k,y^{i,k}$, and $y^{i,k}_j$ as this guarantees that the
construction
\begin{equation}
\label{eq:laplace_projection_constructive}
p_{jk} = (\tau^{i,j}(p) \vee y^{i,k}_j) \cap (\tau^{i,k}(p) \vee y^{i,j}_k)
\subset l^i \vee p
\end{equation}
yields a well-defined map
\begin{equation*}
\tau^i : \RP^3 \setminus \left\{ l^i \right\} \to \RP^3,
\quad
p \mapsto p_{jk}.
\end{equation*}
To see the existence of $\pi^i_{jk}$,
note that the above construction may be reversed and that $\tau^i$ maps
lines to lines according to
\begin{equation*}
l \mapsto l_{jk} = (\tau^{i,j}(l) \vee y^{i,k}_j) \cap (\tau^{i,k}(l) \vee
y^{i,j}_k).
\end{equation*}
Thus $\tau^i$ is a bijection between open subsets of $\RP^3$ that maps lines to
lines and therefore the restriction of a projective transformation $\tilde
\tau^i$ due to the fundamental theorem of real projective geometry.
Accordingly, $\pi^i_{jk} = \tilde \tau^i(\pi^i)$ is the well-defined plane we
are looking for.
\end{proof}

\begin{remark}
\label{rem:compatible_planes_via_concurrency}
Observe that in the above proof the point $q^i = \pi^i \cap \pi^i_j \cap
\pi^i_k$ is a fixed point of $\tau^i$ and hence $q^i \in \pi^i_{jk}$.
Accordingly, the plane $\pi^i_{jk}$ may be constructed as $\pi^i_{jk} = q^i \vee
e^i_{jk}$ and $q^i$ becomes the common intersection of the
four characteristic lines of the direction $i$, that is, $q^i = a^i \cap a^i_j
\cap a^i_{jk} \cap a^i_k$. This shows that for an SC-cube in fact \emph{all} planes
supporting the characteristic conics of its supercyclidic faces associated with
one and the same direction $i$ contain the point $q^i$.
\end{remark}

\begin{corollary}
\label{cor:compatible_planes}
Let $\pi^i,\pi^i_j,\pi^i_k$, and $\pi^i_{jk}$ be four planes through the
respective $i$-edges of a Q-cube $(x,x_i,x_{ij},x_{123})$ in $\RP^3$, such that
each plane contains exactly one edge. Then the central projections
\eqref{eq:laplace_projection} through the (collinear) Laplace points commute in
the sense of \eqref{eq:commuting_projections} if and only if the four planes are
concurrent.
\end{corollary}

\begin{theorem}[Properties of SC-cubes]
\label{thm:sc_cube_properties}
Characteristic conics around an SC-cube close up and bound (unique) SC-patches as
indicated in Fig.~\ref{fig:SC_cube_families}. Patches obtained this way consitute
three families $\mathcal{F}^i$, $i=1,2,3$, where the notation is such that
patches of the family $\mathcal{F}^i$ interpolate smoothly between opposite
cyclidic faces of the cube with respect to the direction $i$.
The induced SC-patches have the following properties:
\begin{enumerate}[(i)]
\item 
For each direction $i$ there exists a unique point $q^i$, which is the common
intersection of the characteristic lines $a^i$ of all patches $\mathcal{F}^j
\cup \mathcal{F}^k$, $\left\{ i,j,k \right\} = \left\{ 1,2,3 \right\}$.
\label{item:sc_cube_properties_common_point}
\item
Patches $f^i \in \mathcal{F}^i$ and $f^j \in \mathcal{F}^j$, $i \ne j$,
intersect along a common characteristic conic that is associated with their
shared net direction.
\label{item:sc_cube_properties_common_conic}
\item
Patches $f^i, \tilde f^i \in \mathcal{F}^i$ are classical fundamental transforms
of each other. In particular, this holds for opposite cyclidic faces of the
initial SC-cube.
\label{item:sc_cube_properties_ftrafo}
\end{enumerate}
\end{theorem}

\begin{figure}[htb]
\begin{center}
\quad\quad\quad 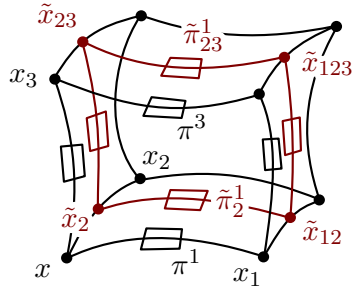 
\end{center}
\caption{SC-cubes induce smooth families of SC-patches.}
\label{fig:SC_cube_families}
\end{figure}

\begin{proof}
The proof of Theorem~\ref{thm:cyclidic_3d_system} shows that going along characteristic
conics of the supercyclidic faces of an SC-cube closes up as indicated in
Fig.~\ref{fig:SC_cube_families}. Generically, one obtains four conic arcs that
intersect in four coplanar vertices\footnote{If opposite boundary curves of a supercyclidic face intersect,
those intersections are singular points of the patch (Remark~\ref{rem:singularities}).
Accordingly, one obtains only three or less vertices when going around the cube if
one passes through such points. However, as the singularities are isolated, the
presented argumentation may be kept valid by incorporation of a continuity
argument.} as the four centers of perspectivity that relate isoparametric points
on the patch boundaries are collinear (cf.  Fig.~\ref{fig:laplace_projection_cube}).
In order to verify that those arcs constitute the boundary of an SC-patch, it
remains to show that
opposite arcs are related by a ``Laplace perspectivity'' with respect
to the constructed vertex quadrilateral. So consider
Fig.~\ref{fig:SC_cube_families} and note that the vertices $x,x_1,x_{13},x_3,\tilde
x_2,\tilde x_{12},\tilde x_{123},\tilde x_{23}$ constitute a Q-cube.  Moreover,
there are unique planes $\pi^1,\tilde \pi^1_2, \tilde \pi^1_{23},\pi^1_3$
through its edges of direction 1 that support the characteristic conics of the
original cyclidic faces and those planes intersect in a point $q^1$ according to
Remark~\ref{rem:compatible_planes_via_concurrency}.  We know already that some of the
characteristic conics in those planes are related by Laplace perspectivities,
namely
\begin{equation*}
\tau^{1,3} : \pi^1 \to \pi^1_3,
\quad
\tilde\tau^{1,2} : \pi^1 \to \tilde\pi^1_2,
\quad
\tilde\tau^{1,2}_3 : \pi^1_3 \to \tilde\pi^1_{23}.
\end{equation*}
Therefore, also the curves in the planes $\tilde\pi^1_2$ and $\tilde\pi^1_{23}$ are
related by the Laplace perspectivity $\tilde\tau^{1,3}_2 : \tilde\pi^1_2 \to
\tilde\pi^1_{23}$ according to Corollary~\ref{cor:compatible_planes}.

We conclude that characteristic conics of the supercyclidic faces of an SC-cube
induce three families of SC-patches, each family interpolating between opposite
faces.  Denote by $\mathcal{F}^i$ the family that interpolates in the direction
$i$. The construction implies that the patches of the family $\mathcal{F}^i$ may
be parametrized smoothly by points $\tilde x^i$ of the boundary arc $b^i$ that connects $x$
and $x_i$. Note that, as a consequence, any
intermediate patch $f^i \in \mathcal{F}^i$ splits the original SC-cube into two
smaller SC-cubes, for which $f^i$ is a common face.

\emph{Proof of \iref{item:sc_cube_properties_common_point}:}
This is an immediate consequence of
Remark~\ref{rem:compatible_planes_via_concurrency}.

\emph{Proof of \iref{item:sc_cube_properties_common_conic}:}
As indicated in Fig.~\ref{fig:SC_cube_fundamental_trafo}, the patches $f^i$ and
$f^j$ determine a pair of corresponding points $p$ and $p_k$ on the opposite
faces of the SC-cube with respect to the common net direction $k$ of $f^i$ and
$f^j$.  Together with the vertices of $f^i$ and $f^j$, the points $p$ and $p_k$
induce a refinement of the original supporting Q-cube into four smaller Q-cubes.
This follows from the aforementioned property that any patch $f^i \in
\mathcal{F}^i$ different from the opposite faces of the SC-cube splits the cube
into two smaller SC-cubes, together with the fact that the characteristic conics
around an SC-cube close up. We conclude that for two of the induced SC-cubes the
vertices $p$ and $p_k$ are connected by a characteristic conic of $f^i$, while
for the other two cubes they are connected by a characteristic conic
of $f^j$. The fact that those conics coincide is again an implication of
Remark~\ref{rem:compatible_planes_via_concurrency} together with
Corollary~\ref{cor:compatible_planes}.

\emph{Proof of \iref{item:sc_cube_properties_ftrafo}:}
It is sufficient to prove the claim for opposite faces of the SC-cube, since any
pair $f^i, \tilde f^i \in \mathcal{F}^i$ may be understood as a pair of opposite
faces after
a possible 2-step reduction of the initial SC-cube to a smaller SC-cube that is
induced by $f^i$ and~$\tilde f^i$.  Now, corresponding points $p$ and $p_k$ on
opposite faces, top and bottom, say, may be identified by using "vertical"
patches as indicated in Fig.~\ref{fig:SC_cube_fundamental_trafo}. Thus property
\iref{item:sc_cube_properties_common_conic} implies that corresponding tangents
in those points intersect due to the tangent cone property for the
vertical patches.
\end{proof}

\begin{figure}[htb]
\begin{center}
 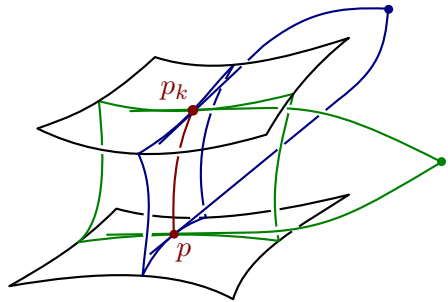 
\end{center}
\caption{Opposite faces of an SC-cube form a fundamental pair.}
\label{fig:SC_cube_fundamental_trafo}
\end{figure}

\paragraph{Supercyclidic coordinate systems}
We say that an SC-cube is \emph{regular} if its six supercyclidic faces do not
contain singularities and opposite faces are disjoint. Note that this implies
that all patches of all three families are regular, which may be derived from
the fact that isolated singularities of SC-patches appear as intersections of
opposite boundary curves, cf. Remark~\ref{rem:singularities}.
This in turn shows that for a regular SC-cube, any two patches $f^i, \tilde f^i \in
\mathcal{F}^i$ are disjoint.  Otherwise there would exist a patch $f$ of another
family through $p \in f^i \cap \tilde f^i$ for which the same reasoning as above
together with a ``cutting down argument'' as in the proof of Theorem~\ref{thm:sc_cube_properties}
would imply that $p$ is a singular point of $f$.  Having observed this, we see
that for a regular SC-cube the property
\iref{item:sc_cube_properties_common_conic} of Theorem~\ref{thm:sc_cube_properties}
gives rise to classical conjugate coordinates on the region that is covered by
the SC-patches of the families $\mathcal{F}^1$, $\mathcal{F}^2$, and
$\mathcal{F}^3$, cf. Fig.~\ref{fig:SC_cube_coordinates}.  Associated with those
coordinates, planarity of vertex quadrilaterals of sub-patches of the coordinate
surfaces induces arbitrary Q-refinements of the original Q-cube. 

\begin{figure}[htb]
\begin{center}
 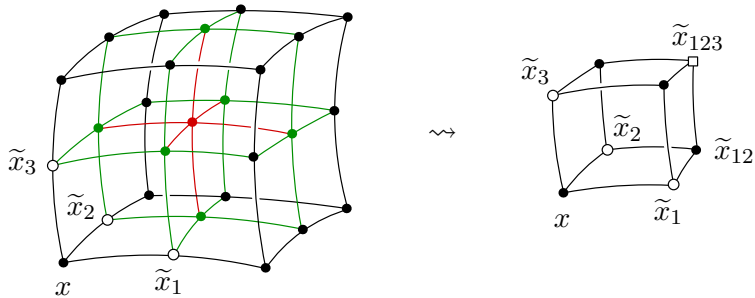 
\end{center}
\caption{Regular SC-cubes induce smooth conjugate coordinates and give rise to
arbitrary Q-refinements of the supporting Q-cube.}
\label{fig:SC_cube_coordinates}
\end{figure}

\paragraph{TCC-reduction of the supercyclidic 3D system}
Theorem~\ref{thm:cyclidic_3d_system} shows that it is possible to equip elementary
quadrilaterals of a 3D Q-net consistently with adapted SC-patches that satisfy
the common boundary condition, such an extension being determined by its
restriction to three coordinate planes (one of each family). We will now
demonstrate that imposition of the TCC-condition in coordinate planes, i.e., to
require that each 2D layer is a 2D supercyclidic net is an admissible reduction
of the underlying 3D system, which in turn shows the existence of 3D
supercyclidic nets. In fact, it turns out that the TCC-reduction is
multidimensionally consistent, which will later give rise to multidimensional
SC-nets as well as fundamental transformations thereof that exhibit the
same Bianchi-permutability properties as fundamental transformations of
classical conjugate nets and their discrete counterparts.

\begin{theorem}
\label{thm:tcc_reduction}
Imposition of the TCC-property on 2D coordinate planes is an admissible
reduction of the SC 3D system. This means, propagation of admissible TCC-Cauchy
data for coordinate planes $\coordsurf{ij}$
according to Theorem~\ref{thm:cyclidic_3d_system} yields a Q-net with SC-patches
adapted to elementary quadrilaterals, such that the common boundary condition is
satisfied and the TCC-property holds in all 2D coordinate planes. In particular,
the reduction is multidimensionally consistent.
\end{theorem}

\begin{proof}
Denote the supporting Q-net by $x$ and the adapted patches by $f$ and assume
TCC-Cauchy data $(x,f)|_{\coordsurf{jk}}$, $1 \le j < k \le m$, to be given
(admissible in the sense that the common boundary condition is satisfied along
coordinate axes).  To begin with, note that~$x$ is uniquely determined by the
Cauchy data $x|_{\coordsurf{ij}}$ as Q-nets are governed by an $m$D consistent
3D system. Accordingly, we may assume that $x$ is given and focus on the Cauchy
data $f|_{\coordsurf{jk}}$ for adapted SC-patches.

The propagation of adapted patches may be understood as propagation of their
defining boundary curves and it has been shown
(Remark~\ref{rem:compatible_planes_via_concurrency} and Corollary~\ref{cor:compatible_planes}) that, on a
3D cube, the propagation of boundary curves by Laplace perspectivities is
consistent if and only if the boundary planes satisfy
\begin{equation}
\label{eq:boundary_plane_propagation_3d}
\pi^i_{jk} = (\pi^i \cap \pi^i_j \cap \pi^i_k) \vee x_{jk} \vee x_{ijk}
= q^i \vee e^i_{jk},
\quad
i \ne j \ne k \ne i.
\end{equation}
Accordingly, it has to be shown that the propagation of boundary planes subject
to \eqref{eq:boundary_plane_propagation_3d} is multidimensionally consistent.
Moreover, we have to assure that the induced lines
\begin{equation*}
t^i = \pi^i_{-i} \cap \pi^i
\end{equation*}
at vertices $x$ constitute a discrete torsal line system. This is
the case if and only if they are compatible with adapted SC-patches, while automatically
implying the TCC-property in coordinate planes (see
Proposition~\ref{prop:properties_supercyclidic_patches}~\iref{item:prop_patches_concurrent_tangent_congruence_quad}
and Lemma~\ref{lem:tcc_characterization}).
Both suggested properties may be verified by a reverse argument: Instead of
considering the Cauchy problem for boundary planes, consider the evolution of
tangents subject to Lemma~\ref{lem:2d_systems}~(L) and for the admissible Cauchy data
$t^i|_{\coordsurf{ij}}$, $j = 1,\ldots,m$. Due to
Corollary~\ref{cor:2d_quad_line_systems} the propagation is multidimensionally consistent
and one obtains a unique fundamental line system $t^i$. It remains to note that
the solution $t^i$ encodes the unique solution 
\begin{equation*}
\pi^i = t^i \vee t^i_i
\end{equation*}
of the original Cauchy problem, for the reason that
\eqref{eq:boundary_plane_propagation_3d} is satisfied by virtue of
Lemma~\ref{lem:fundamental_cubes}~\iref{item:lem_fun_cubes_concurrent_planes_any}.
\end{proof}

\section{$m$D supercyclidic nets and their fundamental transformations}
\label{sec:md_supercyclidic_nets}

The previous considerations motivate the following definition of $m$-dimensional
supercyclidic nets, whose existence is verified by Theorem~\ref{thm:tcc_reduction}.

\begin{definition}[$m$D supercyclidic net]
\label{def:sc_net_md_from_2d}
Let $x:\Z^m \to \RP^3$, $m \ge 2$, be a Q-net and assume that $f: \left\{
\text{2-cells of } \Z^m \right\} \to \left\{ \text{supercyclidic patches in }
\RP^3 \right\}$ describes SC-patches, which are adapted to the elementary
quadrilaterals of $x$ and satisfy the common boundary condition. The pair
$(x,f)$ is called an \emph{$m$D supercyclidic net} if the restriction to any
coordinate surface is a 2D supercyclidic net, that is, if the TCC-property is
satisfied in each 2D coordinate plane.
\end{definition}

\paragraph{Discrete data for an $m$D supercyclidic net}
Analogous to the 2-dimensional case, an $m$-dimensional supercyclidic
net is completely encoded in the following data:

\begin{enumerate}[(i)]
\item The supporting Q-net
$x: \Z^m \to \RP^3$.
\item The discrete torsal tangent systems
$t^i : \Z^m \to \linespace^3$, $i = 1,\ldots,m$.
\item The boundary curves $b^i : \Z^m \to \arcs$, $i = 1,\ldots,m$, with $b^i$
tangent to $t^i$ at $x$ and to $t^i_i$ at $x_i$ and such that for each
quadrilateral opposite boundary curves are perspective from the corresponding
Laplace point.
\end{enumerate}

See also Remark~\ref{rem:discrete_data_2d} on the fully discrete description of
boundary curves in terms of points and lines.

\paragraph{Cauchy data for an $m$D supercyclidic net}
According to Theorem~\ref{thm:tcc_reduction}, Cauchy data for an $m$D supercyclidic net
is given by a collection $(x^{ij},f^{ij})$ of 2D supercyclidic nets for the 2D
coordinate planes $\coordsurf{ij}$, $1 \le i < j \le m$, which has to be
admissible in the sense that the common boundary condition is satisfied along
coordinate lines. Independent Cauchy data in turn is given by
Cauchy data for such compatible 2D supercyclidic nets and we conclude that the
following constitutes Cauchy data for an $m$D supercyclidic net.

\begin{enumerate}[(i)]
\item
\label{item:cauchy_data_sc_qnet}
Cauchy data for the supporting Q-net $x : \Z^m \to \RP^3$, e.g.,
\begin{equation*}
x|_{\coordsurf{ij}},\quad 1 \le i < j \le m.
\end{equation*}
\item
\label{item:cauchy_data_sc_tangents}
Cauchy data for the discrete torsal tangent systems $t^i$, $i = 1,\ldots,m$,
e.g.,
\begin{equation*}
t^i|_{\coordsurf{j}},\quad i,j=1,\ldots,m.
\end{equation*}
\item
\label{item:cauchy_data_sc_splines}
One adapted spline of conic segments for each direction, e.g.,
\begin{equation*}
b^i|_{\coordsurf{i}},\quad i=1,\ldots,m,
\end{equation*}
such that $b^i$ is tangent to $t^i$ at $x$ and to $t^i_i$ at $x_i$.
\end{enumerate}

\paragraph{Fundamental transformations of supercyclidic nets}
Motivated by the theories of cyclidic and hyperbolic nets, we define
fundamental transformations of supercyclidic nets by combination of the
according smooth and discrete fundamental transformations of the involved
SC-patches and supporting Q-nets, respectively.

\begin{definition}[F-transformation of supercyclidic nets]
\label{def:sc_net_f_trafo}
Two $m$D supercyclidic nets $(x,f)$ and $(x^+,f^+)$ are called \emph{fundamental
transforms} of each other if the supporting Q-nets $x$ and $x^+$ form a discrete
fundamental pair and each pair of corresponding SC patches $f,f^+$ forms a
classical fundamental pair.
\end{definition}

It turns out that two $m$D supercyclidic nets are fundamental transforms if and
only if they may be embedded as two consecutive layers of an $(m+1)$-dimensional
supercyclidic net. While the possibility of such an embedding implies the
fundamental relation according to
Theorem~\ref{thm:sc_cube_properties}~\iref{item:sc_cube_properties_ftrafo}, the
converse has to be shown. It is not difficult to see (think of Cauchy data for
$m$D supercyclidic nets) that this is an implication of the following
Proposition~\ref{prop:SC_cube_and_ftrafos}.

\begin{proposition}
\label{prop:SC_cube_and_ftrafos}
Consider a 3D cube $Q$ with planar faces and SC-patches $f,f^+$ adapted to one
pair of opposite faces of $Q$. Then the following statements are equivalent:
\begin{enumerate}[(i)]
\item
\label{item:sc_cube_trafos_transforms}
The patches $f$ and $f^+$ are classical fundamental transforms of each other.
\item
\label{item:sc_cube_trafos_perspective}
Corresponding boundary curves $b$ and $b^+$ of $f$ and $f^+$ are perspective
from the associated Laplace point $(x \vee x^+) \cap (y \vee y^+)$ of the face
of $Q$ that consists of the vertices $x,y$ of $b$ and $x^+,y^+$ of $b^+$.
\item
\label{item:sc_cube_trafos_extension}
The configuration can be extended to an SC-cube. The extension is uniquely
determined by the choice of the four missing boundary planes that have to be
concurrent, and a conic segment that connects corresponding vertices  in
one of those planes.
\end{enumerate}
\end{proposition}

For the proof of Proposition~\ref{prop:SC_cube_and_ftrafos} we will use the following

\begin{lemma}
\label{lem:fundamental_curves}
Let $f,f^+ : U \subset \R^2 \to \RP^3$ be smooth conjugate nets that are
fundamental transforms of each other and let $b$ and $b^+$ be corresponding
parameter lines, that is, tangents to $b$ and $b^+$ in corresponding points
intersect.  Further denote by $c$ the curve that consists of intersection points
of the conjugate tangents along $b$ and $b^+$.  If $b$ and $b^+$ both are planar
and $c$ is regular, then $c$ is also a planar curve. Moreover, the supporting
planes of $b,b^+$, and $c$ intersect in a line.
\end{lemma}

\begin{proof}[Proof of Lemma~\ref{lem:fundamental_curves}]
Denote by $t$ and $t^+$ the tangents to $b$ and $b^+$ and by $g$ and $g^+$ the
conjugate tangents to $f$ and $f^+$. As $b$ and $b^+$ are planar, corresponding
tangents $t$ and $t^+$ intersect along the line $l$ that is the intersection of
the supporting planes of $b$ and $b^+$. In order to see that the curve $c$
composed of the intersection points $g \cap g^+$ is planar, note the following:
The tangent planes $\pi$ to $f$ along $b$ are of the form $t \vee g$ and
analogous for the tangent planes $\pi^+$ of $f^+$ one has $\pi^+ = t^+ \vee
g^+$. So $\pi \cap \pi^+ = h = (t \cap t^+) \vee (g \cap g^+)$ and we conclude
$h \cap l = t \cap t^+$. Moreover, the lines $h$ are the tangents of $c$, since
the planes $\pi$ along $b$ envelop a torsal ruled
surface $\sigma$ whose rulings are precisely the conjugate tangents
$g$.\footnote{This is a classical, purely geometric description of conjugate
tangents to a smooth surface (cf. \cite{PottmannWallner:2001:ComputationalLineGeometry}).}
Accordingly, $c$ is a curve on $\sigma$ and thus in each point of $c$ its
tangent is contained in the corresponding tangent plane $\pi$ of $\sigma$. Of
course the same argumentation holds if we consider the planes $\pi^+$ and hence
the tangents to $c$ must be of the form $\pi \cap \pi^+$.  Finally, for regular
$c$ the planarity follows from the fact that continuity of a regular curve
$\gamma$ implies that $\gamma$ is planar if and only if there exists a line that
intersects all of its tangents. Clearly, the supporting plane of $c$ also passes
through $l$.
\end{proof}

\begin{proof}[Proof of Proposition~\ref{prop:SC_cube_and_ftrafos}]
A Q-cube with one conic segment per edge that connects the incident vertices
encodes an SC-cube if and only if for each quadrilateral the conic segments that
connect the vertices of opposite edges are ``Laplace perspective'',
cf.~Fig.~\ref{fig:laplace_projection_cube}. Thus,
\iref{item:sc_cube_trafos_perspective} is necessary for
\iref{item:sc_cube_trafos_extension}. It is also sufficient since the suggested
construction in \iref{item:sc_cube_trafos_extension} always yields an SC-cube
according to Corollary~\ref{cor:compatible_planes} and actually covers all possible
extension as a consequence of the proof of Theorem~\ref{thm:cyclidic_3d_system}.  On the
other hand, \iref{item:sc_cube_trafos_extension} $\implies$
\iref{item:sc_cube_trafos_transforms} is covered by
Theorem~\ref{thm:sc_cube_properties}~\iref{item:sc_cube_properties_ftrafo}
and it remains to show \iref{item:sc_cube_trafos_transforms} $\implies$
\iref{item:sc_cube_trafos_perspective}.

So consider two corresponding boundary curves $b$ and $b^+$ of the SC-patches
$f$ and $f^+$ that connect four coplanar vertices $x,y,x^+,y^+$.  As the patches
$f$ and~$f^+$ are fundamental transforms, the tangents to $b$ and $b^+$ in
isoparametric points $p$ and $p^+$ intersect along the line $l$ that is the
intersection of the two supporting planes $\pi$ and $\pi^+$ of $b$ and $b^+$,
respectively.  On the other hand, also the conjugate tangents $g$ and $g^+$ in
$p$ and $p^+$ intersect in $q = g \cap g^+$ because of the fundamental relation
in the other direction and those points $q$ trace a curve $c$. Moreover, all
conjugate tangents along $b$ pass through the apex $a$ of the tangency cone
$\mathcal{C}$ to $f$ along $b$ and analogously all conjugate tangents along
$b^+$ pass through the apex $a^+$ of the tangency cone $\mathcal{C}^+$ to $f^+$
along $b^+$, cf. Fig.~\ref{fig:ftrafo_proof}. The curve $c$ is contained in the
intersection $\mathcal{C} \cap \mathcal{C}^+$ and thus is regular. As $b$ and
$b^+$ are planar, the conditions of Lemma~\ref{lem:fundamental_curves} are met and we
may conclude that $c$ is planar (actually, a conic arc) and contained in a plane
$\pi_c$ that passes through the line $l$.

\begin{figure}[htb]
 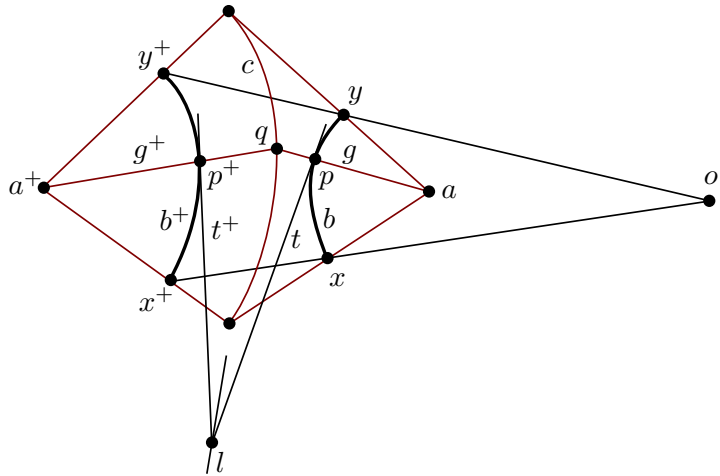 
\caption{Two boundary curves of a fundamental pair of SC-patches with coplanar
vertices $x,y,x^+,y^+$.}
\label{fig:ftrafo_proof}
\end{figure}

It is obvious that isoparametric points on $b$ and $c$ are perspective from $a$
(along the generators of $\mathcal{C}$) and analogously $b^+$ and $c$ are
perspective from $a^+$. Accordingly, those perspective relations between curves
are restrictions of the perspective projective transformations
\begin{equation*}
\begin{array}{lll}
\tau_a : \pi \to \pi_c, & \quad  p \mapsto q \quad & (\text{central projection
through } a)\\
\tau_{a^+} : \pi_c \to \pi^+, & \quad  q \mapsto p^+ \quad & (\text{central
projection through } a^+).
\end{array}
\end{equation*}
Further denote by $\tau = \tau_{a^+} \circ \tau_a$ the composition of those
perspectivities so that
\begin{equation*}
\tau : \pi \to \pi^+, \quad p \mapsto p^+
\end{equation*}
identifies corresponding points on $b$ and $b^+$. The fact that $\tau$ is also a
perspectivity may be seen as follows. First note, that $\tau$ is determined by
four points in general position in $\pi$ and their images in $\pi^+$.  We know
already $\tau(x) = x^+$ as well as $\tau(y) = y^+$ and, moreover, that those
corresponding points are perspective from $o = (x \vee x^+) \cap (y \vee y^+)$.
It remains to observe that $l = \pi \cap \pi^+$ consists of fixed points of
$\tau$ because of $l \subset \pi_c$ and to consider two further points $r,s \in
l$ such that $x,y,r,s$ are in general position.
\end{proof}

\begin{remark*}
The center of perspectivity $o$ for $\tau$ lies on the line $a \vee a^+$. This
follows from the fact that each triple $p,p^+,q$ of corresponding points is
contained in a plane that is spanned by intersecting generators $g$ and $g^+$ of
$\mathcal{C}$ and $\mathcal{C}^+$, respectively. All planes of that kind are
contained in the pencil of planes through the line that is spanned
by the apices of $\mathcal{C}$ and $\mathcal{C}^+$, that is, $a \vee a^+$.
\end{remark*}

\paragraph{Construction of fundamental transforms}
Let $s=(x,f)$ be an $m$D supercyclidic net. According to our previous
considerations, the construction of a fundamental transform $s^+=(x^+,f^+)$ of
$s$ corresponds to the extension of $s$ to a 2-layer $(m+1)$-dimensional
supercyclidic net. Any such extension may be obtained as follows.  First,
construct an F-transform $x^+$ of $x$, that is, extend $x$ to a 2-layer
$(m+1)$-dimensional Q-net $X:\Z^m \times \left\{ 0,1 \right\} \to \RP^3$ with
$X(\cdot,0) = x$ so that $x^+ = X(\cdot,1)$.  Cauchy data for such an extension
is given by the values of $x^+$ along the coordinate axes $\coordsurf{i}$,
$i=1,\ldots,m$, which has to satisfy the condition that the quadrilaterals
$(x,x_i,x^+_i,x^+)$ are planar.  Further, the section on Cauchy data for an $m$D
supercyclidic net shows that the remaining relevant data are the tangents
of $s^+$ at one point $x^+_0$, which have to intersect the corresponding
tangents of $s$ at the corresponding point $x_0$.  The boundary curves of $s^+$
are then obtained as perspective images of the boundary curves of $s$.

\begin{remark*}
In the above construction there are some degrees of freedom left for the
construction of compatible ``vertical'' patches in the resulting 2-layer
$(m+1)$D SC-net, but this does not affect the net $s^+$.
\end{remark*}

\enlargethispage{\baselineskip}

\paragraph{Permutability properties}
Our previous considerations of $m$D supercyclidic nets immediately yield the
following permutability properties for F-transforms of supercyclidic nets.
Corollary~\ref{cor:permutability_scnets} is a literal translation of the
analogous Theorem~\ref{thm:permutability_qnets} for Q-nets and follows from the
fact that the missing tangents of the involved supercyclidic nets are uniquely
determined by their supporting Q-nets (subject to the $m$D consistent 2D system
that describes the extension of Q-nets to torsal line systems, cf.
Corollary~\ref{cor:2d_quad_line_systems}).

\begin{corollary}[Permutability properties of F-transformations of supercyclidic nets]
\label{cor:permutability_scnets}
\hfill
\begin{enumerate}[(i)]
\item 
Let $s=(x,f)$ be an $m$-dimensional supercyclidic net, and let $s^{(1)}$ and
$s^{(2)}$ be two of its F-transforms. Then there exists a 2-parameter family
of SC-nets $s^{(12)}$ that are F-transforms of both $s^{(1)}$ and
$s^{(2)}$. The corresponding vertices of the four SC-nets $s$, $s^{(1)}$,
$s^{(2)}$ and $s^{(12)}$ are coplanar. The net $s^{(12)}$ is uniquely determined
by one of its vertices.
\item
Let $s=(x,f)$ be an $m$-dimensional supercyclidic net. Let $s^{(1)}$, $s^{(2)}$ and $s^{(3)}$ be
three of its F-transforms, and let three further SC-nets $s^{(12)}$,
$s^{(23)}$ and $s^{(13)}$  be given such that $s^{(ij)}$ is a simultaneous
F-transform of $s^{(i)}$ and $s^{(j)}$. Then generically there exists
a unique SC-net $s^{(123)}$ that is an F-transform of $s^{(12)}$,
$s^{(23)}$ and $s^{(13)}$. The net $s^{(123)}$ is uniquely determined by the
condition that for every permutation $(ijk)$ of $(123)$ the corresponding
vertices of $s^{(i)}$, $s^{(ij)}$, $s^{(ik)}$ and $s^{(123)}$ are coplanar.
\end{enumerate}
\end{corollary}

\section{Frames of supercyclidic nets via projective reflections}
\label{sec:frames}

\begin{definition}[Tangent systems and frames of supercyclidic nets]
Let $x:\Z^m \to \RP^3$ be the supporting Q-net of a supercyclidic net and $t^i :
\Z^m \to \linespace^3$, $i=1,\ldots,m$, be the tangents to the supercyclidic
patches at vertices of $x$. We call
\[(t^1,\ldots,t^m) : \Z^m \to \underbrace{\linespace^3 \times
\ldots \times \linespace^3}_{m \text{ times}}\]
the \emph{tangent system} of the supercyclidic net and $(t^1,\ldots,t^m)(z)$ the
\emph{tangents at $x(z)$}.  For $m \in \left\{ 2,3 \right\}$ we refer to the
tangents as \emph{frames} of the supercyclidic net.
\end{definition}

It turns out that the integrable structure behind the tangent systems of
supercyclidic nets (in the sense of Corollary~\ref{cor:2d_quad_line_systems}) has a nice
expression in terms of \emph{projective reflections}
that are associated with the edges of the supporting Q-net. In fact, the
uncovered integrable system on projective reflections governs the simultaneous
extension of a single Q-net to multiple fundamental line systems also in the
multidimensional case and in that sense is more general than the theory of
supercyclidic nets which is rooted in 3-space.

Before going into details, let us recall that any pair of disjoint subspaces
$U,V \subset \RP^N$ with $U \vee V = \RP^N$ induces a unique projective
reflection $f: \RP^N \to \RP^N$ as follows. Points of the union $U \cup V$ are
defined to be fixed points of $f$. For all other points $x$ there exists a
unique line $l$ through $x$ that intersects $U$ and $V$, which yields points $p
= U \cap l$ and $q = V \cap l$. The image $f(x)$ is then defined by the
cross-ratio condition
\begin{equation*}
\crr(x,p,f(x),q) = -1.
\end{equation*}
Note that any projective reflection is an involution, $f^2 = \id$, since
$\crr(x,p,f(x),q) = -1$ if and only if $\crr(f(x),p,x,q) = -1$.

Relevant for our purpose is the particular subclass of projective reflections
that we denote by
\begin{equation*}
\pr(N) = \left\{ \text{Projective reflections in } \RP^N \text{ 
induced by a point } o \text{ and a hyperplane } \pi \not\ni o \right\}.
\end{equation*}

Now the starting point for the following considerations is the observation that
for a generic 3D supercyclidic net any two frames $T=(t^1,t^2,t^3)$ and
$T_i=(t^1_i,t^2_i,t^3_i)$ at adjacent vertices $x$ and $x_i$ are related by a
unique projective reflection $f^i \in \pr(3)$, that is
\begin{equation}
\label{eq:pr_frames}
f^i(T) = T_i
\iff
f^i(T_i) = T.
\end{equation}
To verify this claim, first note that \eqref{eq:pr_frames} implies $f^i(x) =
x_i$.  Further let $f^i \sim (o^i,\pi^i) \in \pr(3)$ be any projective
reflection with $f^i(x) = x_i$ and denote $p^j = t^j \cap \pi^i$. As $x \ne x_i$
we have $x \not\in \pi$ and thus
\begin{equation*}
f^i (t^j) = f^i(x \vee p^j) = x_i \vee p^j.
\end{equation*}
Therefore, \eqref{eq:pr_frames} implies that the $p^j$ have to be the three focal
points $p^j = t^j \cap t^j_i$, $j = 1,2,3$, and $\pi^i = p^1 \vee p^2 \vee p^3$
is uniquely determined. On the other hand, the condition $f^i(x) = x_i$
determines $o^i \in x \vee x_i$ uniquely according to $\crr(x,q^i,x_i,o^i) =
-1$, where $q^i = \pi^i \cap (x \vee x_i)$ is distinct from $x$ and $x_i$.

The next step is to consider frames $T,T_i,T_j,T_{ij}$ at the four vertices
$x,x_i,x_j,x_{ij}$ of a quadrilateral. One obtains
four induced projective reflections $f^i,f^i_j,f^j,f^j_i \in \pr(3)$ for which
\begin{equation*}
(f^j_i \circ f^i) (T) = T_{ij} = (f^i_j \circ f^j) (T)
\end{equation*}
and in particular
\begin{equation}
\label{eq:commuting_reflections_vertices}
f^i(x) = x_i,
\quad
f^j(x) = x_j,
\quad
f^i_j(x_j) =
f^j_i(x_i) = x_{ij}.
\end{equation}
Clearly, $f^i,f^j$, and $T$ determine the frames $T_i,T_j$, and $T_{ij}$
subject to $T_i = f^i(T), T_j = f^j(T)$, and
Lemma~\ref{lem:2d_systems}~(L).  Accordingly, the projective reflections
$f^i_j$ and $f^j_i$ may be constructed from $f^i$ and $f^j$ but the construction
depends on $T$.  However, it turns out that this constructive propagation
\begin{equation}
\label{eq:reflection_propagation}
(f^i,f^j) \mapsto (f^i_j,f^j_i)
\end{equation}
is independent of $T$ modulo \eqref{eq:commuting_reflections_vertices}.  Before
we show this, note that the previous considerations and described constructions
also apply for the simultaneous extension of a Q-net $x : \Z^m \to \RP^N$ to
fundamental line systems $t^1,\ldots,t^N$ and yield unique projective
reflections that relate the lines at adjacent vertices. Accordingly, we will now
go beyond the setting of supercyclidic nets in $\RP^3$ and consider the multidimensional
case in $\RP^N$.  Moreover, it is evident that the frame dependent construction of
reflections is $m$D consistent because of the multidimensional consistency of
the extension of Q-nets to fundamental line systems, cf.
Corollary~\ref{cor:2d_quad_line_systems}.

In order to show that the propagation \eqref{eq:reflection_propagation} depends
only on the supporting Q-net, still inducing an $m$D consistent 2D system with
variables in $\pr(N)$ on edges, we will need the following

\begin{lemma}
\label{lem:reflection_planes_intersection}
Let $Q=(x,x_1,x_{12},x_2)$ be a planar quadrilateral in $\RP^N$ and
$(t^i,t^i_1,t^i_{12},t^i_2)$, $i=1,\ldots,N$, be $N$ simultaneous extensions of
$Q$ to torsal line congruence quadrilaterals subject to
Lemma~\ref{lem:2d_systems}~(L). Further denote by $\pi^j = \bigvee_i (t^i \cap
t^i_j)$ the hyperplane spanned by focal points that are associated with the edge
$x \vee x_j$ (cf. Fig.~\ref{fig:reflection_planes}) and let $X = \pi^1 \cap \pi^2
\cap \pi^1_2 \cap \pi^2_1$.  Then $\codim(X) = 2$.
\end{lemma}

\begin{figure}[htb]
\begin{center}
 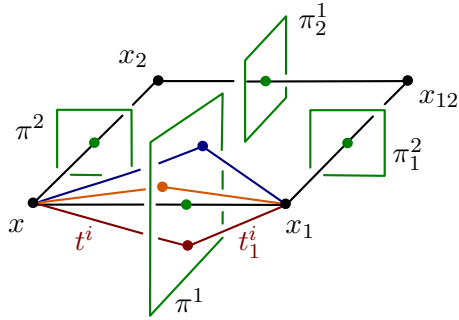 
\end{center}
\caption{Hyperplanes spanned by focal points of line congruence quadrilaterals
that extend the same planar quadrilateral.}
\label{fig:reflection_planes}
\end{figure}

\begin{proof}
For $1 \le i < j \le N$ the eight corresponding lines
$t^i,t^j,\ldots,t^i_{12},t^j_{12}$ form a line complex cube for which $Q$ is a
focal quadrilateral. Now consider the four lines $l^{1,ij} = (t^i \cap t^i_1)
\vee (t^j \cap t^j_1) \subset \pi^1$, $l^{2,ij} \subset \pi^2$, $l^{1,ij}_2
\subset \pi^1_2$, and $l^{2,ij}_1 \subset \pi^2_1$ that are spanned by the
remaining focal points of the complex cube. Lemma~\ref{lem:fundamental_cubes} implies
that those four lines intersect in one point $p^{ij} = l^{1,ij} \cap l^{2,ij}
\cap l^{1,ij}_2 \cap l^{2,ij}_1 \in X$ due to the planarity of $Q$.  As always,
we assume that the data is generic so that $\tilde X = \bigvee_{i \ne j} p^{ij}$
is a subspace of codimension 1 in $\pi^1 = \bigvee_{i \ne j} l^{1,ij}$ and thus
has codimension 2 in $\RP^N$. Now $\tilde X \subset X \subset \pi^1 \cap \pi^2$
together with $\codim(\pi^1 \cap \pi^2) = 2$ shows $X = \tilde X$ and hence
proves the lemma.
\end{proof}

\begin{theorem}[Extension of Q-nets to
fundamental line systems via projective reflections]
\label{thm:projective_reflections}
Let $(x,x_1,x_{12},x_2)$ be a planar quadrilateral in $\RP^N$, $N \ge 3$, and
let $f^1,f^2 \in \pr(N)$ be projective reflections with $f^i(x) = x_i$. Then
there exist unique projective reflections $f^1_2,f^2_1 \in \pr(N)$, which are
determined by the conditions
\begin{equation}
\label{eq:pr_propagation}
f^i_j(x_j) = x_{12}
\text{ and }
(f^2_1 \circ f^1) (l) = (f^1_2 \circ f^2) (l)
\text{ for all } l \in \linespace^N \text{ with } x \in l.
\end{equation}
Further, let $x : \Z^m \to \RP^N$, $N \ge 3$, be a generic Q-net. Then the 
induced propagation
\begin{equation*}
\tau :
\Z^m \times \pr(N) \times \pr(N) \to \pr(N) \times \pr(N), \quad
(f^i,f^j) \mapsto (f^i_j,f^j_i)
\end{equation*}
of projective reflections is $m$D consistent, that is, admissible Cauchy data
$f^i|_{\coordsurf{i}}$, $i = 1,\ldots,m$, (mapping adjacent vertices onto each
other) determines all remaining projective reflections uniquely.
\end{theorem}

\begin{proof}
First observe that the maps $f^i|_{\coordsurf{i}}$, $i = 1,\ldots,m$, together
with one line $t(0) \ni x(0)$ are Cauchy data for the ($m$D consistent)
extension of $x$ to a fundamental line system. Accordingly, a frame
$T(0)=(t^1,\ldots,t^N)(0)$ at $x(0)$ yields well defined frames at every vertex
of $x$. Those frames in turn yield one well-defined projective reflection for
each edge as described before: The hyperplane $\pi^i$ of the reflection $f^i
\sim (o^i,\pi^i)$ associated with the edge $(x,x_i)$ is spanned by the
corresponding focal points, $\pi^i = \bigvee_{j=1,\ldots,N} (t^j \cap t^j_i)$,
which in turn allows to construct the center $o^i$ from the condition $f^i(x) =
x_i$. It remains to show that the construction is independent of $T(0)$, that
is, the propagation $(f^i,f^j) \mapsto (f^i_j,f^j_i)$ according to
\eqref{eq:pr_propagation} is well-defined.

So consider a planar quadrilateral $Q=(x,x_1,x_{12},x_2)$ in $\RP^N$ with one
generic frame $T=(t^1,\ldots,t^N)$ at $x$ and let $f^1,f^2 \in \pr(N)$ with
$f^i(x) = x_i$ be given. One obtains frames $T_1 = f^1(T)$ at $x_1$ and $T_2 =
f^2(T)$ at $x_2$ and Lemma~\ref{lem:2d_systems}~(L) yields a unique frame
$T_{12}$ at $x_{12}$.  As described above this induces projective reflections
$f^1_2,f^2_1$ such that
\begin{equation}
\label{eq:pr_commuting_frames}
(f^2_1 \circ f^1) (T) = T_{12} = (f^1_2 \circ f^2) (T)
\end{equation}
and we have to show that the so obtained maps satisfy \eqref{eq:pr_propagation}.
By construction we have $f^i_j(x_j) = x_{12}$. On the other hand, the second condition
$(f^2_1 \circ f^1) (l) = (f^1_2 \circ f^2) (l)$ for all
$l \in \linespace^N(x) = \left\{ l \in \linespace^N \mid x \in l \right\}$
may be rewritten as
\begin{equation}
\label{eq:pr_identity}
(f^2 \circ f^1_2 \circ f^2_1 \circ f^1) (l) = l,
\end{equation}
since projective reflections are involutions.\footnote{Note that the identity
\eqref{eq:pr_identity} on sets does not imply $f^2 \circ f^1_2 \circ f^2_1 \circ
f^1 = \id$. The composition could also be a homothety from an affine point of
view.} 
In order to verify \eqref{eq:pr_identity} we identify lines through vertices of
$Q$ with their intersections with the hyperplanes of the projective reflections
according to
\begin{equation*}
\begin{array}{lclclcl}
x & \in & l & \quad \leftrightarrow \quad & p & = & l \cap \pi^1\\
x_1 & \in & l_1 & \quad \leftrightarrow \quad & p_1 & = & l_1 \cap \pi^2_1\\
x_{12} & \in & l_{12} & \quad \leftrightarrow \quad & p_{12} & = & l_{12} \cap \pi^1_2\\
x_2 & \in & l_2 & \quad \leftrightarrow \quad & p_2 & = & l_2 \cap \pi^2\\
\end{array}
\end{equation*}
After this identification, the images of lines through the vertices of $Q$ under
the considered projective reflections correspond to the images of their
representing points under certain central projections:
\begin{align*}
f^1 : \linespace^N(x) \to \linespace^N(x_1) & \quad \leftrightarrow \quad
\tau^1 : \pi^1 \to \pi^2_1 \quad (\text{central projection through } x_1)\\
f^2_1 : \linespace^N(x_1) \to \linespace^N(x_{12}) & \quad \leftrightarrow \quad
\tau^2_1 : \pi^2_1 \to \pi^1_2 \quad (\text{central projection through } x_{12})\\
f^1_2 : \linespace^N(x_{12}) \to \linespace^N(x_2) & \quad \leftrightarrow \quad
\tau^1_2 : \pi^1_2 \to \pi^2 \quad (\text{central projection through } x_2)\\
f^2 : \linespace^N(x_2) \to \linespace^N(x) & \quad \leftrightarrow \quad
\tau^2 : \pi^2 \to \pi^1 \quad (\text{central projection through } x)
\end{align*}
The identity \eqref{eq:pr_identity} may then be rewritten as
\begin{equation}
\label{eq:tau_identity}
\tau = \tau^2 \circ \tau^1_2 \circ \tau^2_1 \circ \tau^1 = \id.
\end{equation}
As $\tau : \pi^1 \to \pi^1$ is a projective transformation and
$\dim(\pi^1)=N-1$, it is sufficient to show that there are $N+1$ fixed points of
$\tau$ in general position. In fact, we already have $N$ fixed points because of
\eqref{eq:pr_commuting_frames}, that is
\begin{equation*}
\tau(p^i) = p^i, \quad i = 1,\ldots,N,
\end{equation*}
for $p^i = t^i \cap \pi^1$.
The existence of one further
fixed point is guaranteed by Lemma~\ref{lem:reflection_planes_intersection}.
\end{proof}

The analytic description of Q-nets with frames at vertices via reflections might
be useful as it turned out to be in the special case of circular nets.  For
circular nets, orthonormal frames at vertices were introduced in
\cite{BobenkoMatthesSuris:2003:OrthogonalSystems} and played an essential role
in the proof of the $C^\infty$-convergence of circular nets to smooth orthogonal
nets. The Dupin cyclidic structure behind was discovered later and gave rise to
the introduction of cyclidic nets in
\cite{BobenkoHuhnen-Venedey:2011:cyclidicNets}.  In that case, the orthonormal
frames at adjacent vertices of the supporting circular net are related by a
reflection in the Euclidean symmetry plane, which is a special instance of a
projective reflection (clearly, cyclidic nets in $\R^3$ are special instances of
supercyclidic nets).  Note that the Euclidean reflections in the cyclidic case
are uniquely determined by the supporting circular net, thus a Dupin cyclidic
net is uniquely determined by its supporting circular net and the frame in one
vertex.

\bibliographystyle{abbrv}

\end{document}